\documentclass[11pt,reqno,english]{amsart}
\title[Long range asymptotics for the quadratic AB NLS]
{Long Range Asymptotics for the Quadratic\\
Aharonov--Bohm NLS}

\usepackage[T1]{fontenc}
\usepackage{lmodern,eucal}
\usepackage[a4paper,margin=1in]{geometry}
\usepackage{amssymb}

\usepackage[
    pdfauthor={Piero D'Ancona and Tohru Ozawa},
    colorlinks=true,
    linkcolor=magenta,
    citecolor=cyan,
    urlcolor=blue,
    hyperfootnotes=false
]{hyperref}

\hypersetup{
  pdftitle={
    Long Range Asymptotics for the Quadratic
    Aharonov--Bohm NLS
  },
  pdfsubject={
    Modified wave operators, second order asymptotics,
    and Friedrichs trace classification
  }
}

\allowdisplaybreaks[4]

\numberwithin{equation}{section}

\newtheorem{theorem}{Theorem}[section]
\newtheorem{corollary}[theorem]{Corollary}
\newtheorem{lemma}[theorem]{Lemma}
\newtheorem{proposition}[theorem]{Proposition}
\theoremstyle{remark}
\newtheorem{remark}[theorem]{Remark}

\theoremstyle{definition}
\newtheorem{definition}[theorem]{Definition}

\DeclareMathOperator{\dom}{dom}
\newcommand{\hyp}[1]{\texorpdfstring{$#1$}{}}
 
\date{August 2026}
\author{Piero D'Ancona}
\address{Piero D'Ancona: Sapienza Universit\`a di Roma,
Dipartimento di Matematica, Piazzale A.~Moro 2,
I-00185 Roma, Italy}
\email{dancona@mat.uniroma1.it}
\author{Tohru Ozawa}
\address{Tohru Ozawa: Waseda University,
Department of Applied Physics, School of Advanced Science and Engineering,
Tokyo, Japan}
\email{txozawa@waseda.jp}
\thanks{The first author is supported by the Sapienza
  University research projects 
  ``Wave dynamics in heterogeneous media''
  (DYNAWAVES), CUP B83C25000880005, and ``Bridging Analysis and
  Computation in Evolutionary PDEs'', 
  and by the Gruppo Nazionale per
  l'Analisi Matematica, la Probabilit\`a e le loro 
  Applicazioni (GNAMPA)
  of the Istituto Nazionale di Alta Matematica (INdAM).
  The second author is supported by JSPS KAKENHI Grant Number 
  24H00024 and JST Moonshot R \& D Grant number GMPJMS25A5.}


\begin{document}

\begin{abstract}
  We study the long range behavior of solutions to $i\partial_tu=H_\alpha u+\lambda|u|u$ on $\mathbb R^2$, where $H_\alpha$ is the Friedrichs realization of the Aharonov--Bohm Hamiltonian with a single pole. The logarithmic phase of the long range ansatz may push a profile out of the domain of $H_{\alpha}$. We characterize profiles that stay in the operator domain by the vanishing of boundary traces at 0 of order $\le \frac 12$; at half flux $\alpha=1/2$, no nonzero trace survives.

  However, every profile in the full domain of $H_{\alpha}$ with small $L^\infty$ amplitude determines a unique global solution with a modified final state, with a remainder rate $t^{-b}$ for all $0<b<1/2+\nu_\alpha$, $\nu_\alpha=\min\{\alpha,1-\alpha\}$. For profiles satisfying the vanishing trace condition, the rate improves to every $0<b<1$. This result is sharp in the sense that, if $\alpha\neq \frac 12$, we can construct profiles with an error of size $t^{-1/2-\nu_\alpha}\log t$, ruling out all faster rates. The upper bound comes from a retarded Strichartz estimate for a residual that is not in $L^2$; the lower bound is an explicit calculation via Hankel transforms.

  For smoother profiles we also compute the first correction, which gives remainder rates with $1<b<2$.
\end{abstract}
\maketitle

\section{Introduction}\label{sec:introduction}

Consider the Aharonov--Bohm potential on 
$\mathbb{R}^{2}\setminus\{0\}$
\begin{equation}
  \label{eq:AB-potential}
  A_\alpha(x)
  =
  \alpha\frac{(-x_2,x_1)}{|x|^2},
  \qquad
  \alpha\in\mathbb R,
\end{equation}
and let $H_\alpha$ be the Friedrichs realization in
$L^2(\mathbb R^2)$ of
\begin{equation}
  \label{eq:AB-Hamiltonian}
  H_\alpha
  =
  \bigl(-i\nabla+A_\alpha\bigr)^2,
\end{equation}
with domain $\dom(H_\alpha)$.
If $\alpha\in\mathbb Z$, the Hamiltonian is gauge
equivalent to the free Laplacian, and after a gauge reduction,
we may assume $0<\alpha<1$ as we shall do in the following.

We are interested in the quadratic, gauge  invariant NLS
\begin{equation}
  \label{eq:main-NLS}
  i\partial_tu
  =
  H_\alpha u+\lambda |u|u,
  \qquad
  \lambda\in\mathbb R.
\end{equation}
For every $u_0\in L^2(\mathbb R^2)$, this equation has a
unique global solution in
$C(\mathbb R;L^2)\cap L^6_{\mathrm{loc}}(\mathbb R;L^3)$;
we briefly recall the well known proof in
Proposition~\ref{prop:global-L2-theory} below.

Since the nonlinearity is long range in two space dimensions,
a nonzero solution is not expected to approach a
fixed linear solution.  Instead, the correct final state
should contain a logarithmic phase.
For $\gamma\in\mathbb R$, we define the nonlinear phase modifier
\begin{equation}
  \label{eq:G-gamma}
  G_\gamma(z)
  =
  z e^{-i\gamma|z|},
  \qquad
  z\in \mathbb{C}.
\end{equation}
The long range ansatz corresponds to the choice
$\gamma=\lambda\log(t)/2$.  
Thus we define the logarithmically modified profile as
\begin{equation}
  \label{eq:intro-W}
  W(t)
  =
  G_{\lambda\log(t)/2}(\phi)
  =
  \phi
  \exp\left(
    -\frac{i\lambda}{2}|\phi|\log t
  \right).
\end{equation}
Note that for $\alpha\in(0,1)$, $G_\gamma$ does not preserve
the whole domain $\dom(H_\alpha)$ in general.  Recall that, for
a fixed radial cutoff
$\chi_{\mathrm{tr}}\in C_c^\infty([0,\infty))$ equal 
to one near zero,
every $\phi\in\dom(H_\alpha)$ has a unique decomposition
\begin{equation}
  \label{eq:full-trace-decomposition}
  \phi
  =
  \psi
  +
  c_0\chi_{\mathrm{tr}} r^\alpha
  +
  c_{-1}\chi_{\mathrm{tr}} r^{1-\alpha}e^{-i\theta},
  \qquad
  \psi\in\mathcal Y,
\end{equation}
where $x = (r \cos \theta, r \sin \theta)$ in $\mathbb{R}^{2}$ 
and $\mathcal Y$ is the minimal full angular domain defined in
\eqref{eq:pole-compatible-space}--\eqref{eq:pole-compatible-norm};
see Lemma~\ref{lem:full-minimal-domain} and
\cite{DG21,Fer24}.
A different choice of the cutoff changes $\psi$ but not the
Friedrichs trace coefficients $c_0(\phi)=c_0$ and
$c_{-1}(\phi)=c_{-1}$ in \eqref{eq:full-trace-decomposition},
which are uniquely determined.
We define a subspace
\begin{equation}
  \label{eq:full-trace-candidate}
  \mathcal D_\alpha^\sharp
  =
  \left\{
    \phi\in\dom(H_\alpha):
    c_m(\phi)=0
    \ \text{for every }m\in\{0,-1\}
    \ \text{with }|m+\alpha|\leq\frac12
  \right\}.
\end{equation}
Thus one trace is allowed in the definition of 
$\mathcal D_\alpha^\sharp$ if $\alpha\neq \frac 12$,
whereas if $\alpha=\frac 12$ both trace coefficients 
are required to vanish.
We endow $\mathcal D_\alpha^\sharp$ with the graph norm of
$H_{\alpha}$. In the following we write
\begin{equation}
  \label{eq:lowest-trace-order}
  \nu_\alpha
  =
  \min\{\alpha,1-\alpha\}.
\end{equation}

\subsection{Main results}

The first result identifies the exact subspace of 
$\dom(H_\alpha)$ preserved by the action of
the logarithmic nonlinear phase modifier.

\begin{theorem}[Domain of the modifier]
\label{the:sharp-full-domain}
  Let $0<\alpha<1$ and
  $\gamma\in\mathbb R\setminus\{0\}$.  For every
  $\phi\in\dom(H_\alpha)$,
  \begin{equation}
    \label{eq:sharp-full-domain-classification}
    G_\gamma(\phi)\in\dom(H_\alpha)
    \quad\Longleftrightarrow\quad
    \phi\in\mathcal D_\alpha^\sharp.
  \end{equation}
  On $\mathcal D_\alpha^\sharp$, the modifier $G_{\gamma}$ 
  preserves both trace coefficients.  
  For each fixed $\gamma$, $G_{\gamma}$ maps
  $\mathcal D_\alpha^\sharp$ continuously into itself for the
  graph topology, and at each fixed $\phi$ the map
  $\gamma\mapsto G_\gamma(\phi)$ is graph continuous.  Moreover,
  \begin{equation}
    \label{eq:full-modifier-graph-growth}
      \|G_\gamma(\phi)\|_{\mathrm{graph}(H_\alpha)}
      :=
      \|G_\gamma(\phi)\|_2
      +
      \|H_\alpha G_\gamma(\phi)\|_2
      \leq
      C_\phi(1+|\gamma|)^2.
  \end{equation}
\end{theorem}

The theorem shows that $G_{\gamma}$ preserves
a Friedrichs trace in channel $m$ exactly when $|m+\alpha|>1/2$; 
at and below the threshold $|m+\alpha|=1/2$, 
that trace must vanish for $G_{\gamma}(\phi)$ to remain in the
domain of the operator.
Graph continuity is optimal in the sense that,
by Proposition~\ref{prop:graph-Lipschitz-failure},
local Lipschitz continuity fails near moving simple zeros.

Nevertheless, profiles belonging to
$\dom(H_\alpha)\setminus\mathcal D_\alpha^\sharp$
are not excluded from the final state construction,
as we prove next. For a forbidden trace, the
quantity $H_\alpha W(t)$ may fail to belong to $L^2$, but its
singular part belongs to an $L^p$ space in the range used by
the retarded Strichartz estimate, and this is enough
to include forbidden traces in the construction,
although with a worse decay rate for the remainder.

We show the construction.
For $\phi\in L^2(\mathbb R^2)$ and $t>0$, define
$u_{\mathrm{ap}}(t,x)$ by
\begin{equation}
  \label{eq:all-channel-uap}
  u_{\mathrm{ap}}(t,x)
  =
  \frac1{2it}
  e^{i|x|^2/(4t)}
  \phi\left(\frac{x}{2t}\right)
  \exp\left[
    -\frac{i\lambda}{2}
    \left|
      \phi\left(\frac{x}{2t}\right)
    \right|
    \log t
  \right].
\end{equation}
To express the same asymptotic in distorted Fourier variables,
let $\mathcal F_\alpha$ denote the unitary transform defined in
\eqref{eq:distorted-transform} below.  Set
\begin{equation}
  \label{eq:u-plus-definition}
  u_+
  =
  \mathcal F_\alpha^{-1}\phi
\end{equation}
and define
\begin{equation}
  \label{eq:distorted-nonlinear-modifier}
  \mathcal N_\alpha(t)u_+
  =
  \mathcal F_\alpha^{-1}
  \biggl[
    \exp\left(
      -\frac{i\lambda}{2}
      |\mathcal F_\alpha u_+|\log t
    \right)
    \mathcal F_\alpha u_+
  \biggr].
\end{equation}
This is the nonlinear modifier which appears in the modified
scattering property \eqref{eq:main-modified-wave-asymptotic}
below, which relates the linear flow with the nonlinearly
perturbed equation \eqref{eq:main-NLS}.

For $1/2<b<1$, define the profile class
\begin{equation}
  \label{eq:rate-dependent-profile-class}
  \mathcal P_{\alpha,b}
  =
  \begin{cases}
    \dom(H_\alpha),
    &
    \frac12<b<\frac12+\nu_\alpha,
    \\ 
    \mathcal D_\alpha^\sharp,
    &
    \frac12+\nu_\alpha\leq b<1.
  \end{cases}
\end{equation}
The next result proves the existence of the modified wave operator
for \eqref{eq:main-NLS}:

\begin{theorem}[Modified final states]
\label{the:main-final-state}
  Let $0<\alpha<1$, $1/2<b<1$, and
  $\lambda\in\mathbb R$.  There is
  $\varepsilon_{b,\alpha,\lambda}>0$ such that for every
  $\phi\in\mathcal P_{\alpha,b}$ satisfying
  \begin{equation}
    \label{eq:main-profile-smallness}
    \|\phi\|_\infty
    <
    \varepsilon_{b,\alpha,\lambda}
  \end{equation}
  there exists a unique global solution
  $u(t)$ of \eqref{eq:main-NLS} satisfying
  \begin{equation}
    \label{eq:main-final-state-bound}
    \sup_{\tau\geq T}
    \tau^b
    \bigl(
      \|u-u_{\mathrm{ap}}\|_{
        L^\infty([\tau,\infty);L^2)}
      +
      \|u-u_{\mathrm{ap}}\|_{
        L^4([\tau,\infty);L^4)}
    \bigr)
    <
    \infty
  \end{equation}
  for some $T\geq1$.  In addition,
  \begin{equation}
    \label{eq:main-modified-wave-asymptotic}
    \left\|
      u(t)
      -
      e^{-itH_\alpha}
      \mathcal N_\alpha(t)u_+
    \right\|_2
    \longrightarrow0
    \qquad
    (t\to\infty).
  \end{equation}
  Thus, on the class
  \begin{equation}
    \label{eq:wave-operator-domain}
    \mathcal U_{\alpha,b,\lambda}
    =
    \left\{
      u_+\in L^2:
      \mathcal F_\alpha u_+
      \in\mathcal P_{\alpha,b},
      \|\mathcal F_\alpha u_+\|_\infty
      <
      \varepsilon_{b,\alpha,\lambda}
    \right\},
  \end{equation}
  the modified nonlinear wave operator
  \begin{equation}
    \label{eq:wave-operator-map}
    \Omega_{\alpha,\lambda}^{\mathrm{mod}}:
    u_+
    \longmapsto
    u(0)
  \end{equation}
  is well defined and injective.
\end{theorem}

Note that only the $L^\infty$ amplitude of the profiles must be 
small; the graph norm may be large but its size
affects only the starting time $T$ and the
constants.  The profiles may meet the pole and may contain
arbitrary angular modes.  If $b<1/2+\nu_\alpha$, every
Friedrichs trace is allowed.  When $\alpha=\frac 12$, 
this covers every
$1/2<b<1$, even though Theorem~\ref{the:sharp-full-domain}
says that the pointwise modifier preserves no nonzero
exceptional trace.  If $\alpha\neq \frac 12$, restricting to
$\mathcal D_\alpha^\sharp$ restores the full range $b<1$;
in this case,
$\mathcal D_\alpha^\sharp$ is strictly larger than
$\mathcal Y$ (see \eqref{eq:pole-compatible-space}) and
contains $\chi r^\alpha$ when $\alpha>1/2$
and $\chi r^{1-\alpha}e^{-i\theta}$ when $\alpha<1/2$.

If $\alpha\neq1/2$, the dichotomy
\eqref{eq:rate-dependent-profile-class} is sharp for the
present ansatz, as the following explicit construction shows.

\begin{theorem}[Sharpness]
\label{the:pure-trace-rate-sharpness}
  Let $0<\alpha<1$, $\alpha\neq1/2$, 
  $\nu=\nu_\alpha$, and let
  $m_\nu=0$ if $0<\alpha<1/2$ and
  $m_\nu=-1$ if $1/2<\alpha<1$, so that
  $|m_\nu+\alpha|=\nu_\alpha$.

  Choose a radial cutoff
  $\chi\in C_c^\infty([0,\infty))$ with $0\leq\chi\leq1$ which
  is one near zero, and let
  \begin{equation}
    \label{eq:pure-trace-profile}
    \phi(r,\theta)
    =
    c\chi(r)r^\nu e^{im_\nu\theta},
    \qquad
    \lambda c\neq0,
  \end{equation}
  so that $\phi\in \dom(H_\alpha)\setminus\mathcal D_\alpha^\sharp$.
  Set $b_0=1/2+\nu/2$.  Assume
  $\|\phi\|_\infty<\varepsilon_{b_0,\alpha,\lambda}$, and let
  $u$ be the corresponding final state solution in
  Theorem~\ref{the:main-final-state}.  Then
  \eqref{eq:main-final-state-bound} fails for every
  \begin{equation}
    \label{eq:pure-trace-forbidden-rate}
    b
    \geq
    \frac12+\nu
  \end{equation}
  and every $T\geq1$.
  More generally, no solution of \eqref{eq:main-NLS} satisfies
  \eqref{eq:main-final-state-bound} for this profile and any
  $b\geq1/2+\nu$.
\end{theorem}

Of course, the previous sharpness statement concerns 
the logarithmic ansatz \eqref{eq:all-channel-uap},
but it does not rule out a faster remainder
after adding higher order corrections.

The next result computes the first term beyond the logarithmic
profile.  
The construction of the first correction is guided by 
higher order expansions for flat critical NLS, notably 
Cazenave–Naumkin~\cite{CN18} and the polynomial logarithmic 
hierarchy of Jendrej–Salvi~\cite{JS26}. Compare with the 
candidate second profile of Masaki–Miyazaki–Uriya~\cite{MMU19}.
We restrict to profiles whose modulus satisfies
$\rho(x)=O(|x|^6)$ near 0, which we call
\emph{second order core} profiles;
see Definition~\ref{def:second-order-core}.  These stronger
hypotheses keep regularity at zeros out of the asymptotic
calculation.  On $\mathbb R^2\setminus\{0\}$, we write the
profiles in polar form $\phi=\rho\omega$, with $\rho\geq0$ and
$|\omega|=1$.

\begin{theorem}[First correction]
\label{the:second-order}
  Fix $1<b<2$ and $\lambda\in\mathbb R$.  Let
  $\phi=\rho\omega$ be a second order core profile in the sense
  of Definition~\ref{def:second-order-core}, and assume that
  $\|\rho\|_\infty$ is sufficiently small.  Put
  $a=\lambda\rho/2$, let
  \begin{equation}
    \label{eq:q-polynomial}
    q(\sigma)
    =
    q_2\sigma^2+q_1\sigma+q_0,
    \qquad
    \sigma\in \mathbb{R},
  \end{equation}
  where the coefficients are determined by
  \eqref{eq:second-order-q-recursion}, and 
  with the notations in \eqref{eq:MD-definition-intro},
  define
  \begin{equation}
    \label{eq:second-order-uap}
    u_{\mathrm{ap}}^{(2)}(t)
    =
    u_{\mathrm{ap}}(t)
    +
    \frac1tM(t)D(t)
    \left[
      e^{-ia\log t}q(\log t)
    \right].
  \end{equation}
  Then \eqref{eq:main-NLS} has a unique global solution satisfying
  \begin{equation}
    \label{eq:second-order-asymptotic}
    \sup_{\tau\geq T}
    \tau^b
    \bigl(
      \|u-u_{\mathrm{ap}}^{(2)}\|_{
        L^\infty([\tau,\infty);L^2)}
      +
      \|u-u_{\mathrm{ap}}^{(2)}\|_{
        L^4([\tau,\infty);L^4)}
    \bigr)
    <
    \infty
  \end{equation}
  for some $T\geq1$. In particular,
  \begin{equation}
    \label{eq:second-order-expansion}
    u(t)
    =
    u_{\mathrm{ap}}(t)
    +
    \frac1tM(t)D(t)
    \left[
      e^{-ia\log t}q(\log t)
    \right]
    +
    O_{L^2}(t^{-b}).
  \end{equation}
\end{theorem}

\subsection{The mechanism}

The two structural identities used in the proof are
\begin{equation}
  \label{eq:intro-factorization}
  e^{-itH_\alpha}
  =
  M(t)D(t)\mathcal F_\alpha M(t)
\end{equation}
and
\begin{equation}
  \label{eq:intro-pseudoconformal}
  (i\partial_t-H_\alpha)M(t)D(t)W
  =
  M(t)D(t)
  \left(
    i\partial_tW-\frac{H_\alpha W}{4t^2}
  \right).
\end{equation}
Here $\mathcal F_\alpha$ is the explicit Aharonov--Bohm
distorted Fourier transform \eqref{eq:distorted-transform},
while
\begin{equation}
  \label{eq:MD-definition-intro}
  M(t)f(x)
  =
  e^{i|x|^2/(4t)}f(x),
  \qquad
  D(t)f(x)
  =
  \frac1{2it}f\left(\frac{x}{2t}\right).
\end{equation}
Both $M(t)$ and $D(t)$ are unitary on $L^2(\mathbb R^2)$.

For $W$ defined by \eqref{eq:intro-W}, set
$u_{\mathrm{ap}}(t)=M(t)D(t)W(t)$ as in \eqref{eq:all-channel-uap}.
The factor $1/2$ in the phase is fixed by the normalization
of $D(t)$.  Indeed,
\begin{equation}
  \label{eq:intro-nonlinearity-scaling}
  |u_{\mathrm{ap}}|u_{\mathrm{ap}}=
  |M(t)D(t)W|M(t)D(t)W
  =
  M(t)D(t)
  \left(\frac1{2t}|W|W\right).
\end{equation}
$W$ solves the pointwise equation
\begin{equation}
  \label{eq:intro-profile-ODE}
  i\partial_tW
  =
  \frac{\lambda}{2t}|W|W.
\end{equation}
Combining \eqref{eq:intro-pseudoconformal} and
\eqref{eq:intro-profile-ODE} we get the exact error
\begin{equation}
  \label{eq:intro-error}
  (i\partial_t-H_\alpha)u_{\mathrm{ap}}
  -
  \lambda|u_{\mathrm{ap}}|u_{\mathrm{ap}}
  =
  -\frac1{4t^2}M(t)D(t)H_\alpha W.
\end{equation}
The thresholds in the main results come from the same
local term. To simplify the discussion,
ignore the fixed angular factor and cutoff, and
focus on a nonzero trace $cr^\mu$, where $0<\mu<1$.  Near the
pole, for $\gamma\neq0$,
\begin{equation*}
  G_\gamma(cr^\mu)
  =
  cr^\mu
  -i\gamma|c|c r^{2\mu}
  +O(r^{3\mu}).
\end{equation*}
The model Bessel operator 
$L_\mu=-\Delta_{r}+\frac{\mu^{2}}{r^{2}}$,
where
$\Delta_r:=\partial_r^2+\frac1r\partial_r$,
annihilates $r^\mu$ and
sends the new term $r^{2\mu}$ to a nonzero multiple of
$r^{2\mu-2}$.  Since the radial measure is $r\,dr$,
\begin{equation*}
  r^{2\mu-2}\in L^2((0,1),r\,dr)
  \quad\Longleftrightarrow\quad
  \mu>\frac12.
\end{equation*}
This is exactly the domain threshold in
Theorem~\ref{the:sharp-full-domain}.

When the trace condition
holds, $\|H_\alpha W(t)\|_2$ grows at most like
$(1+\log t)^2$, so the error in \eqref{eq:intro-error} is
integrable in $L^1_t(0,\infty; L^2_x)$.

When $\mu\leq1/2$, the $L^2$ test fails, but
$r^{2\mu-2}$ belongs locally to $L^p$ exactly when
$p<(1-\mu)^{-1}$.  For the lowest trace order
$\mu=\nu_\alpha$, Lemma~\ref{lem:full-domain-weak-image}
therefore places $H_\alpha W$ in $L^2+L^p$ for
$1<p<(1-\nu_\alpha)^{-1}$.  After dilation, the $L^p$ part is
bounded by $t^{-3+2/p}(1+\log t)^2$.  The dual Strichartz time
exponent gives the tail
\begin{equation}
  \label{eq:intro-mixed-residual-tail}
  O\left(
    t^{-\delta_p}(1+\log t)^2
  \right),
  \qquad
  \delta_p
  =
  \frac32-\frac1p.
\end{equation}
Letting $p$ approach $(1-\nu_\alpha)^{-1}$ explains the range
$b<1/2+\nu_\alpha$ in
Theorem~\ref{the:main-final-state}.  Finally, the high
frequency Hankel calculation in
Subsection~\ref{sec:pure-trace-rate-sharpness} shows that, on
a pure forbidden trace, this same local term produces the
corresponding lower bound.
In other words, the trace that breaks the operator
domain also sets the best decay rate for the logarithmic ansatz.

One long range difficulty remains.  The linear term in
\begin{equation*}
  |u_{\mathrm{ap}}+z|
  (u_{\mathrm{ap}}+z)
  -
  |u_{\mathrm{ap}}|u_{\mathrm{ap}}
\end{equation*}
has size $t^{-1}\|\phi\|_\infty |z|$.  Its time integral is
logarithmically divergent without decay of $z$.  We close the
final state problem in a weighted Strichartz space with
$\|z(t)\|_2=O(t^{-b})$, $1/2<b<1$.  Smallness of
$\|\phi\|_\infty$ absorbs this linear long range term, 
and the quadratic term in $z$ is smaller because $b>1/2$.
At the next order, the recursion
\eqref{eq:second-order-q-recursion} cancels the 
$t^{-2}$ error in \eqref{eq:intro-error}, and the resulting
residual is $O(t^{-3}(1+\log t)^4)$ in $L^2$.

\begin{remark}[Negative times]
\label{rem:negative-times}
  We stated all results for $t\to+\infty$, but the
  case $t\to-\infty$ is perfectly symmetric.
  If $u$ solves \eqref{eq:main-NLS} with flux $\alpha$, then
  $v(t,x)=\overline{u(-t,x)}$ solves
  $i\partial_tv=H_{-\alpha}v+\lambda|v|v$.
  For $0<\alpha<1$, the flux $-\alpha$ is gauge equivalent
  to $1-\alpha$.  Thus all the above results have past analogues,
  and the decay rates are exactly the same
  because $\nu_{1-\alpha}=\nu_\alpha$.

  More precisely, for a past profile $\phi$, 
  we use the same formulas for $M(t)$
  and $D(t)$ at $t<0$ and set
  \begin{equation*}
    W_-(t)
    =
    \phi
    \exp\left(
      \frac{i\lambda}{2}|\phi|\log|t|
    \right),
    \qquad
    u_{\mathrm{ap}}^-(t)
    =
    M(t)D(t)W_-(t).
  \end{equation*}
  The nonlinear phase changes sign because
  $i\partial_tW_-=\lambda|W_-|W_-/(2|t|)$.
  The sharpness and the first correction statements transfer
  immediately.
\end{remark}

\subsection{Relation with the literature and organization}

The free long range mechanism originates in
Ginibre--Ozawa~\cite{GO93}; see also \cite{MM18} for critical
homogeneous nonlinearities and \cite{HN06} for the domain and
range of the flat modified wave operator.  The Aharonov--Bohm
representation
and decay theory used here comes from~\cite{FFFP13,FFFP15}; 
see also \cite{GYZZ21,FSWZZ26} for related distorted Fourier and
intertwining results.  Modified scattering for related inverse
square and inhomogeneous models was developed in
\cite{AIMMU21,GX20,KM25,KSW25}.  The selfadjoint realizations
of the single flux Aharonov--Bohm operator were classified by
Adami--Teta~\cite{AT98}; however, throughout this paper we
use only the Friedrichs realization.  The Bessel operator
domain decomposition used in the channel theorem is due to
Derezi\'nski--Georgescu~\cite{DG21}; the corresponding
Aharonov--Bohm boundary expansion is reviewed in
Fermi~\cite{Fer24}.  Existing nonlinear
Aharonov--Bohm results such as \cite{ZZ18,MS24} concern
different nonlinearities or scattering regimes.
For the flat critical equation, higher asymptotic expansions
were obtained for specially oscillatory, nonvanishing data in
\cite{CN18}; related second-profile constructions are given in
\cite{MMU19,JS26}.
We mention that the cited long range constructions work with 
profile classes for which the logarithmic modifier remains 
in the operator domain, and do not address Aharonov--Bohm profiles 
such that $G_\gamma(\phi)\notin\operatorname{dom}(H_\alpha)$,
whose residual must instead be treated in \(L^2+L^p\).
The distinction between operator domain failure and a mixed
$L^2+L^p$ residual is the new mechanism behind
Theorem~\ref{the:main-final-state}.

We conclude with a brief outline of the paper.
Section~\ref{sec:AB-Hamiltonian} collects the spectral
notation, exact linear factorization, Strichartz estimates, and
global $L^2$ theory.
Section~\ref{sec:pseudoconformal} derives the magnetic
pseudoconformal identity and the residual estimate.
The sharp one channel and full angular domain theorems are
proved in Section~\ref{sec:sharp-channel-domains}, together
with the weak full domain image used for forbidden traces.
Section~\ref{sec:all-channel} proves the modified final state,
identifies the nonlinear wave operator, proves sharpness on a
pure forbidden trace, and constructs the first correction.
Appendix~\ref{sec:one-channel} gives a direct Weber--Hankel
normalization of the channel flow.

\section{Linear preliminaries and global \hyp{L^2} theory}
\label{sec:AB-Hamiltonian}

This section fixes notation and collects the standard linear
facts used in the rest of the paper.
Details are included only where the exact
normalization or choice of exponents matters later.

\subsection{Angular decomposition}

In polar coordinates $x=(r\cos\theta,r\sin\theta)$,
\begin{equation}
  \label{eq:AB-polar}
  H_\alpha
  =
  -\partial_r^2-\frac1r\partial_r
  +
  \frac1{r^2}
  \bigl(-i\partial_\theta+\alpha\bigr)^2.
\end{equation}
For $m\in\mathbb Z$, let
\begin{equation}
  \label{eq:mu-m}
  \mu_m
  =
  |m+\alpha|
\end{equation}
and define the Bessel operator
\begin{equation}
  \label{eq:Bessel-operator}
  L_{\mu_m}
  =
  -\partial_r^2-\frac1r\partial_r
  +
  \frac{\mu_m^2}{r^2}
\end{equation}
in $L^2((0,\infty),r\,dr)$, with its Friedrichs
realization.  If
\begin{equation}
  \label{eq:channel-function}
  u(r,\theta)
  =
  e^{im\theta}v(r),
\end{equation}
then
\begin{equation}
  \label{eq:channel-linear}
  H_\alpha u
  =
  e^{im\theta}L_{\mu_m}v.
\end{equation}

\subsection{Hankel transform}

For $\mu\geq0$, define
\begin{equation}
  \label{eq:Hankel-transform}
  (\mathcal H_\mu f)(\rho)
  =
  \int_0^\infty
  J_\mu(r\rho)f(r)r\,dr,
  \qquad
  \rho>0.
\end{equation}
The Hankel transform $\mathcal H_\mu$ is a selfadjoint unitary
operator from $L^2(r\,dr)$ to $L^2(\rho\,d\rho)$ and
\begin{equation}
  \label{eq:Hankel-diagonalization}
  \mathcal H_\mu L_\mu f
  =
  \rho^2\mathcal H_\mu f.
\end{equation}
It also satisfies the dual identity
\begin{equation}
  \label{eq:Hankel-dual}
  \mathcal H_\mu(r^2f)
  =
  L_\mu\mathcal H_\mu f,
\end{equation}
where the operator on the right acts in the $\rho$ variable.

\subsection{Distorted Fourier transform and linear estimates}
\label{sec:factorization}

Decompose $f\in L^2(\mathbb R^2)$ in spherical harmonics
\begin{equation}
  \label{eq:Fourier-series}
  f(r,\theta)
  =
  \sum_{m\in\mathbb Z}
  f_m(r)e^{im\theta}
\end{equation}
and define
\begin{equation}
  \label{eq:distorted-transform}
  (\mathcal F_\alpha f)(\rho,\theta)
  =
  \sum_{m\in\mathbb Z}
  e^{-i\pi\mu_m/2}
  (\mathcal H_{\mu_m}f_m)(\rho)
  e^{im\theta},
  \qquad
  \mu_m
  =
  |m+\alpha|.
\end{equation}
This is a unitary operator on $L^2(\mathbb R^2)$.  On the
transform side, write $H_\alpha^{(\rho,\theta)}$ for the same
differential expression as $H_\alpha$ with radial variable
$\rho$.  The transform obeys
\begin{equation}
  \label{eq:distorted-intertwining}
  \mathcal F_\alpha H_\alpha
  =
  \rho^2\mathcal F_\alpha
\end{equation}
and
\begin{equation}
  \label{eq:distorted-dual-intertwining}
  \mathcal F_\alpha |x|^2
  =
  H_\alpha^{(\rho,\theta)}\mathcal F_\alpha.
\end{equation}

The propagator of $H_{\alpha}$ admits an explicit Dollard type
factorization in terms of the distorted Fourier transform
$\mathcal{F}_{\alpha}$, which is the channel form of
the representation formula in
Fanelli--Felli--Fontelos--Primo~\cite{FFFP13}.  The direct
Hankel proof in Appendix~\ref{sec:channel-factorization} fixes
every phase and the factor $2$ in the nonlinear logarithmic
correction.

\begin{proposition}[Full factorization]
\label{prop:full-factorization}
  For every $t>0$,
  \begin{equation}
    \label{eq:full-factorization}
    e^{-itH_\alpha}
    =
    M(t)D(t)\mathcal F_\alpha M(t).
  \end{equation}
\end{proposition}

\begin{proof}
  Apply the channel factorization
  \eqref{eq:channel-factorization}, proved in
  Appendix~\ref{sec:channel-factorization}, to every coefficient
  in \eqref{eq:Fourier-series}.  Its phase is exactly the
  phase in \eqref{eq:distorted-transform}.  Summation in
  $L^2(\mathbb R^2)$ gives \eqref{eq:full-factorization}.
\end{proof}

\begin{proposition}[Strichartz estimates]
\label{prop:Strichartz}
  Let $(q,r)$ satisfy
  \begin{equation}
    \label{eq:admissible}
    2\leq q,r\leq\infty,
    \qquad
    \frac2q+\frac2r=1,
    \qquad
    (q,r)\neq(2,\infty).
  \end{equation}
  Then
  \begin{equation}
    \label{eq:homogeneous-Strichartz}
    \|e^{-itH_\alpha}f\|_{L^q_tL^r_x}
    \leq
    C_{\alpha,q,r}\|f\|_2.
  \end{equation}
  The corresponding retarded inhomogeneous estimates hold.
  In particular, if $1<p<2$ and
  \begin{equation}
    \label{eq:mixed-Strichartz-exponents}
    q_p'
    =
    \frac{2p}{3p-2},
  \end{equation}
  then
  \begin{equation}
    \label{eq:mixed-retarded-Strichartz}
    \left\|
      \int_t^{\sup I}
      e^{-i(t-s)H_\alpha}G(s)\,ds
    \right\|_{L^\infty_{t}(I;L^2)
      \cap L^4_{t}(I;L^4)}
    \leq
    C_{\alpha,p}
    \|G\|_{L^{q_p'}(I;L^p)}.
  \end{equation}
  In particular, for every interval $I$ and every
  $G\in L^1(I;L^2)$,
  \begin{equation}
    \label{eq:retarded-Strichartz}
    \left\|
      \int_t^{\sup I}
      e^{-i(t-s)H_\alpha}G(s)\,ds
    \right\|_{L^\infty_{t}(I;L^2)
      \cap L^4_{t}(I;L^4)}
    \leq
    C_\alpha\|G\|_{L^1(I;L^2)}.
  \end{equation}
\end{proposition}

\begin{proof}
  The Aharonov--Bohm propagator has the sharp estimate
  \begin{equation}
    \label{eq:AB-dispersive}
    \|e^{-itH_\alpha}f\|_\infty
    \leq
    C_\alpha |t|^{-1}\|f\|_1.
  \end{equation}
  This is Theorem~1.9 of~\cite{FFFP13}.
  The Keel--Tao theorem~\cite{KT98}, unitarity, and the usual
  duality argument give the homogeneous and 
  inhomogeneous estimates
  \eqref{eq:mixed-retarded-Strichartz} and
  \eqref{eq:retarded-Strichartz}.  For
  \eqref{eq:mixed-retarded-Strichartz}, the pair
  $(2p/(2-p),p')$ is admissible and its time dual is
  $q_p'=2p/(3p-2)$.  Use this source pair and the two target
  pairs $(\infty,2)$ and $(4,4)$.  Taking instead the source
  pair $(\infty,2)$ gives
  \eqref{eq:retarded-Strichartz}.
\end{proof}

\subsection{Global \hyp{L^2} theory}
\label{sec:L2-theory}

For completeness, we recall the standard construction
of global solutions for \eqref{eq:main-NLS}.

\begin{proposition}[Global $L^2$ wellposedness]
\label{prop:global-L2-theory}
  For every $u_0\in L^2(\mathbb R^2)$ and
  $\lambda\in\mathbb R$, equation \eqref{eq:main-NLS} has a
  unique global solution
  \begin{equation}
    \label{eq:L2-solution-class}
    u\in C(\mathbb R;L^2)
    \cap
    L^6_{\mathrm{loc}}(\mathbb R;L^3).
  \end{equation}
  Moreover,
  \begin{equation}
    \label{eq:mass-conservation}
    \|u(t)\|_2
    =
    \|u_0\|_2.
  \end{equation}
\end{proposition}

\begin{proof}
  The pair $(6,3)$ is Schr\"odinger admissible in dimension
  two.  On a bounded interval $I$, Strichartz and H\"older give
  \begin{equation}
    \label{eq:local-subcritical-estimate}
    \||u|u\|_{L^{6/5}(I;L^{3/2})}
    \leq
    |I|^{1/2}
    \|u\|_{L^6(I;L^3)}^2.
  \end{equation}
  The analogous difference estimate follows from
  \begin{equation*}
    \bigl||u|u-|v|v\bigr|
    \leq
    (|u|+|v|)|u-v|.
  \end{equation*}
  A contraction in
  $C(I;L^2)\cap L^6(I;L^3)$ gives local existence and
  uniqueness, with a lifespan depending only on
  $\|u_0\|_2$.

  For smooth data, multiplication of
  \eqref{eq:main-NLS} by $\overline u$, integration, and taking
  the imaginary part gives \eqref{eq:mass-conservation}, because
  $H_\alpha$ is selfadjoint and $\lambda$ is real.
  Approximation extends the identity to $L^2$ data.  The local
  lifespan can therefore be restarted uniformly, which proves
  global existence.  We may also follow the conservation law
  argument in Ozawa~\cite{Ozawa06}.
\end{proof}

\section{The pseudoconformal computation}
\label{sec:pseudoconformal}

The exact identity in this section is the core of the nonlinear
argument.

\begin{lemma}[Magnetic pseudoconformal identity]
\label{lem:pseudoconformal}
  Let $0<\alpha<1$, let $I\subset(0,\infty)$ be an interval,
  and assume
  $W\in C(I;\dom(H_\alpha))\cap C^1(I;L^2)$, where the first
  space has the graph topology.  Then, in
  $\mathcal D'(I\times\mathbb R^2)$,
  \begin{equation}
    \label{eq:pseudoconformal-identity}
    (i\partial_t-H_\alpha)
    M(t)D(t)W(t)
    =
    M(t)D(t)
    \left(
      i\partial_tW(t)
      -
      \frac1{4t^2}H_\alpha W(t)
    \right).
  \end{equation}
\end{lemma}

\begin{proof}
  We first assume that $W$ is smooth and supported away from
  the pole.  Put $y=x/(2t)$ and
  $P_\alpha=-i\nabla+A_\alpha$.  The two relevant geometric
  identities are
  \begin{equation}
    \label{eq:A-homogeneous}
    A_\alpha(x)
    =
    \frac1{2t}A_\alpha(y)
  \end{equation}
  and
  \begin{equation}
    \label{eq:A-tangential}
    y\cdot A_\alpha(y)
    =
    0.
  \end{equation}
  Since
  $-i\nabla_x e^{i|x|^2/(4t)}
  =
  (x/(2t))e^{i|x|^2/(4t)}$, we have
  \begin{equation}
    \label{eq:first-covariant-calculation}
    P_{\alpha,x}
    \bigl(M(t)W(y)\bigr)
    =
    M(t)
    \left(
      yW+\frac1{2t}P_{\alpha,y}W
    \right).
  \end{equation}
  Applying $P_{\alpha,x}$ once more and using
  \eqref{eq:A-tangential} gives
  \begin{equation}
    \label{eq:second-covariant-calculation}
    H_{\alpha,x}\bigl(M(t)W(y)\bigr)
    =
    M(t)\Bigl[
      |y|^2W
      +
      \frac1t y\cdot P_{\alpha,y}W
      -
      \frac{i}{t}W
      +
      \frac1{4t^2}H_{\alpha,y}W
    \Bigr].
  \end{equation}
  On the other hand, differentiating the dilation factor,
  the modulation, and $y=x/(2t)$ gives
  \begin{equation}
    \label{eq:time-calculation}
    i\partial_t\bigl(M(t)D(t)W(t,y)\bigr)
    =
    M(t)D(t)\Bigl[
      -\frac{i}{t}W
      +
      |y|^2W
      +
      i\partial_tW
      -
      \frac{i}{t}y\cdot\nabla_yW
    \Bigr].
  \end{equation}
  Because $y\cdot P_{\alpha,y}=-iy\cdot\nabla_y$,
  subtraction of \eqref{eq:second-covariant-calculation}
  from \eqref{eq:time-calculation} proves
  \eqref{eq:pseudoconformal-identity} in this case.

  It remains to cross the pole.  Let
  \begin{equation}
    \label{eq:full-Friedrichs-core}
    \mathcal C_\alpha^{\mathrm F}
    =
    C_c^\infty(\mathbb R^2\setminus\{0\})
    +
    \operatorname{span}
    \left\{
      \chi r^\alpha,
      \chi r^{1-\alpha}e^{-i\theta}
    \right\}
  \end{equation}
  where $\chi=\chi_{\mathrm{tr}}$ is a fixed radial cutoff
  equal to 1 near 0.
  Lemma~\ref{lem:full-minimal-domain} shows that
  $\mathcal C_\alpha^{\mathrm F}$ is an operator core for the
  Friedrichs realization.  The calculation above also applies
  distributionally to either trace vector.  Indeed, its local
  form is $r^\mu e^{im\theta}$ with $\mu>0$, and the boundary
  terms on $r=\varepsilon$ are $O(\varepsilon^\mu)$.
  They therefore vanish as $\varepsilon\downarrow0$, so no
  distribution supported at the pole is produced.

  Thus, for every fixed $V\in\mathcal C_\alpha^{\mathrm F}$,
  \eqref{eq:pseudoconformal-identity} holds for all $t>0$ with 
  $W(t)\equiv V$. Approximation in the
  $H_\alpha$ graph norm extends it to every
  $V\in\dom(H_\alpha)$: the factors $M(t)D(t)$ are unitary,
  and both $V$ and $H_\alpha V$ converge in $L^2$.  Finally,
  the $L^2$ product rule for the time dependent profile adds
  $iM(t)D(t)\partial_tW$.  Graph continuity of $W$ makes the
  right hand side locally continuous in $L^2$ and proves the
  claimed identity.
\end{proof}

The following nonlinear scaling identity is a routine check.

\begin{lemma}[Nonlinear scaling]
\label{lem:nonlinear-scaling}
  For every measurable $W$ and $t>0$,
  \begin{equation}
    \label{eq:nonlinear-scaling}
    |M(t)D(t)W|M(t)D(t)W
    =
    M(t)D(t)
    \left(\frac1{2t}|W|W\right).
  \end{equation}
\end{lemma}


\begin{definition}[Admissible profile]
\label{def:admissible-profile}
  A function $\phi\in L^2\cap L^\infty$ is admissible
  for $(H_\alpha,\lambda)$ if, with
  \begin{equation}
    \label{eq:W-definition}
    W(t)
    =
    \phi
    \exp\left(
      -\frac{i\lambda}{2}|\phi|\log t
    \right),
    \qquad t\geq1,
  \end{equation}
  one has
  $W\in C([1,\infty);\dom(H_\alpha))$ for the graph topology
  and
  \begin{equation}
    \label{eq:profile-admissibility}
    \|H_\alpha W(t)\|_2
    \leq
    C_\phi(1+\log t)^2.
  \end{equation}
\end{definition}

\begin{definition}[Profile space]
\label{def:pole-compatible-space}
  Let
  \begin{equation}
    \label{eq:pole-compatible-space}
    \mathcal Y
    =
    \operatorname{cl}_{\|\cdot\|_{\mathcal Y}}
    C_c^\infty(\mathbb R^2\setminus\{0\}),
  \end{equation}
  where
  \begin{equation}
    \label{eq:pole-compatible-norm}
    \|\phi\|_{\mathcal Y}
    =
    \|\phi\|_{H^2}
    +
    \left\|
      \frac{\nabla\phi}{|x|}
    \right\|_2
    +
    \left\|
      \frac{\phi}{|x|^2}
    \right\|_2.
  \end{equation}
  Thus every element of $\mathcal{Y}$ belongs to $H^2$ 
  and has the two displayed
  weighted derivatives.  The space contains every annular
  $H^2$ profile and, for example, every function
  $|x|^2\psi$ with $\psi\in C_c^\infty(\mathbb R^2)$.
\end{definition}

\begin{lemma}[Compatible profiles]
\label{lem:concrete-profiles}
  Let $\phi\in\mathcal Y$ and $R=\|\phi\|_\infty$.  
  Then $\phi$ is admissible and
  \begin{equation}
    \label{eq:generic-profile-bound}
    \|H_\alpha W(t)\|_2
    \leq
    C_\alpha
    \bigl[
      1
      +
      |\lambda|R\log t
      +
      \lambda^2R^2(\log t)^2
    \bigr]
    \|\phi\|_{\mathcal Y}.
  \end{equation}
\end{lemma}

\begin{proof}
  Put $\sigma=\log t$, $\gamma=\lambda\sigma/2$, and regard
  $G_\gamma$ from \eqref{eq:G-gamma} as a map from
  $\mathbb R^2$ to $\mathbb R^2$.  Although
  $G_\gamma$ need not be twice continuously differentiable at
  zero, it belongs to $W^{2,\infty}_{\mathrm{loc}}$ and
  \begin{align}
    \label{eq:G-first-derivative}
    \sup_{|z|\leq R}|DG_\gamma(z)|
    &\leq
    C(1+|\gamma|R),
    \\
    \label{eq:G-second-derivative}
    \mathop{\rm ess\,sup}_{|z|\leq R}
    |D^2G_\gamma(z)|
    &\leq
    C\bigl(
      |\gamma|+\gamma^2R
    \bigr).
  \end{align}
  Indeed, away from zero these estimates follow by
  differentiating $z e^{-i\gamma|z|}$.  The first derivative
  has a limit at zero, while the second derivatives are
  bounded there.  Approximation at the origin gives the stated
  weak derivative bounds.

  The Sobolev chain rule, first for smooth approximations and
  then by passage to the limit, gives
  \begin{align}
    \label{eq:G-composition-first}
    \|\nabla G_\gamma(\phi)\|_2
    &\leq
    C(1+|\gamma|R)\|\nabla\phi\|_2,
    \\
    \label{eq:G-composition-second}
    \|D^2G_\gamma(\phi)\|_2
    &\leq
    C(1+|\gamma|R)\|D^2\phi\|_2
    +
    C\bigl(
      |\gamma|+\gamma^2R
    \bigr)
    \|\nabla\phi\|_4^2.
  \end{align}
  In two dimensions, the Gagliardo--Nirenberg estimate
  \begin{equation}
    \label{eq:GN-gradient}
    \|\nabla\phi\|_4^2
    \leq
    C\|\phi\|_\infty\|D^2\phi\|_2
  \end{equation}
  turns \eqref{eq:G-composition-second} into
  \begin{equation}
    \label{eq:G-H2-bound}
    \|G_\gamma(\phi)\|_{H^2}
    \leq
    C\bigl(
      1+|\gamma|R+\gamma^2R^2
    \bigr)
    \|\phi\|_{H^2}.
  \end{equation}
  The first derivative bound also holds pointwise almost
  everywhere, and hence
  \begin{equation}
    \label{eq:G-weighted-gradient}
    \left\|
      \frac{\nabla G_\gamma(\phi)}{|x|}
    \right\|_2
    \leq
    C(1+|\gamma|R)
    \left\|
      \frac{\nabla\phi}{|x|}
    \right\|_2.
  \end{equation}
  Moreover, $|G_\gamma(\phi)|=|\phi|$.  On the punctured plane,
  \begin{equation}
    \label{eq:Halpha-expanded}
    H_\alpha
    =
    -\Delta
    -
    2iA_\alpha\cdot\nabla
    +
    |A_\alpha|^2.
  \end{equation}
  Therefore
  \begin{equation}
    \label{eq:Halpha-weighted-bound}
    \|H_\alpha W(t)\|_2
    \leq
    \|\Delta W(t)\|_2
    +
    2|\alpha|
    \left\|
      \frac{\nabla W(t)}{|x|}
    \right\|_2
    +
    \alpha^2
    \left\|
      \frac{W(t)}{|x|^2}
    \right\|_2.
  \end{equation}
  Estimates \eqref{eq:G-H2-bound} and
  \eqref{eq:G-weighted-gradient} prove the required estimate.
  Choose $\phi_n\in C_c^\infty(\mathbb R^2\setminus\{0\})$
  converging to $\phi$ in $\mathcal Y$.  The same composition
  estimates, together with continuity of $G_\gamma$ in the
  $\mathcal Y$-topology (easily checked), imply
  $G_\gamma(\phi_n)\to G_\gamma(\phi)$ in $\mathcal Y$.
  Hence $G_\gamma(\phi)\in\mathcal Y$.
  In particular, $W(t)/|x|\in L^2$: use the last term in
  \eqref{eq:pole-compatible-norm} on $|x|<1$ and the $L^2$
  norm on $|x|\geq1$.  Together with
  $\nabla W(t)\in L^2$, this shows that $W(t)$ lies in the
  magnetic form domain \eqref{eq:Friedrichs-form-domain}.
  Estimate
  \eqref{eq:Halpha-weighted-bound} shows that the distribution
  $H_\alpha W(t)$ lies in $L^2$.
  By the Friedrichs form characterization (see~\cite{FFFP15}),
  a function in the magnetic form domain whose distributional
  $H_\alpha$-image belongs to $L^2$ lies in
  $\dom(H_\alpha)$.  Hence $W(t)\in\dom(H_\alpha)$.
  Applying the same composition estimates to
  $\gamma_n\to\gamma$ gives continuity in the graph norm.
  Thus $t\mapsto W(t)$ is graph continuous.
\end{proof}

\begin{proposition}[Approximate solution]
\label{prop:approximate-solution}
  Let $\phi$ be admissible in the sense of
  Definition~\ref{def:admissible-profile}, and define
  \begin{equation}
    \label{eq:uap-definition}
    u_{\mathrm{ap}}(t)
    =
    M(t)D(t)W(t).
  \end{equation}
  Then
  \begin{equation}
    \label{eq:residual-definition}
    R(t)
    :=
    (i\partial_t-H_\alpha)u_{\mathrm{ap}}
    -
    \lambda|u_{\mathrm{ap}}|u_{\mathrm{ap}}
    =
    -\frac1{4t^2}M(t)D(t)H_\alpha W(t).
  \end{equation}
  In particular,
  \begin{equation}
    \label{eq:residual-L2}
    \|R(t)\|_2
    \leq
    \frac{C_\phi}{4t^2}(1+\log t)^2.
  \end{equation}
\end{proposition}

\begin{proof}
  Direct differentiation of \eqref{eq:W-definition} gives
  \begin{equation}
    \label{eq:W-ODE}
    i\partial_tW
    =
    \frac{\lambda}{2t}|W|W.
  \end{equation}
  Since $\phi\in L^2\cap L^\infty$, equation
  \eqref{eq:W-ODE} gives
  $\partial_tW\in C([1,\infty);L^2)$.  
  By admissibility, $W$ satisfies the hypotheses of
  Lemma~\ref{lem:pseudoconformal}, so the
  magnetic pseudoconformal identity applies directly.
  Combining Lemmas~\ref{lem:pseudoconformal} and
  ~\ref{lem:nonlinear-scaling} gives
  \eqref{eq:residual-definition}.  Since $M(t)D(t)$ is unitary
  on $L^2$, \eqref{eq:residual-L2} follows from
  \eqref{eq:profile-admissibility}.
\end{proof}

\section{The sharp nonlinear modifier domain}
\label{sec:sharp-channel-domains}

The nonlinear logarithmic phase need not preserve the full
Friedrichs operator domain.  
To study this effect, we first solve the problem in one
Bessel channel, then we combine the channel decomposition
with a uniform minimal domain estimate to get the full
angular classification.

\subsection{Exact trace classification}

For the one channel classification, introduce the operator
\begin{equation}
  \label{eq:radial-unitary}
  (Uf)(r)
  =
  r^{1/2}f(r),
  \qquad
  r>0.
\end{equation}
Thus $U$ is unitary from $L^2(r\,dr)$ to $L^2(dr)$ and
for $0<\mu<1$ 
\begin{equation}
  \label{eq:radial-conjugation}
  UL_\mu U^{-1}
  =
  -\partial_r^2
  +
  \frac{\mu^2-1/4}{r^2}.
\end{equation}
Let
\begin{equation}
  \label{eq:minimal-channel-domain}
  \mathcal D^{\mathrm{ch}}_{\mu,0}
  =
  U^{-1}H^2_0(\mathbb R_+),
  \qquad
  H^2_0(\mathbb R_+)
  =
  \{v\in H^2(0,\infty):v(0)=v'(0)=0\}.
\end{equation}
As a set, $\mathcal D^{\mathrm{ch}}_{\mu,0}$ is independent of
$\mu$; we endow it with the graph norm of $L_\mu$.
The Bessel domain theorem of
Derezi\'nski--Georgescu~\cite{DG21} gives, for any cutoff
$\chi\in C_c^\infty([0,\infty))$ equal to one near zero,
the following decomposition of the Friedrichs domain:
\begin{equation}
  \label{eq:Friedrichs-channel-domain}
  \dom(L_\mu)
  =
  \mathcal D^{\mathrm{ch}}_{\mu,0}
  \mathbin{\dotplus}
  \mathbb C\chi r^\mu.
\end{equation}
See also Fermi~\cite{Fer24} for the corresponding
Aharonov--Bohm Friedrichs asymptotics.  The sum is topological
in the graph norm $\|f\|_2+\|L_\mu f\|_2$: both the minimal
part and the trace coefficient depend continuously on $f$.
In particular, each $f\in\dom(L_\mu)$ has a unique expansion
(for a fixed cutoff $\chi$)
\begin{equation}
  \label{eq:channel-trace-decomposition}
  f
  =
  g+c\chi r^\mu,
  \qquad
    g\in\mathcal D^{\mathrm{ch}}_{\mu,0},
  \quad c\in\mathbb C.
\end{equation}

The next lemma shows that the minimal remainder produces no 
additional obstruction, in the sense that if
$g\in\mathcal D^{\mathrm{ch}}_{\mu,0}$ then
\begin{equation*}
  G_\gamma(c\chi r^\mu+g)=G_\gamma(c\chi r^\mu)+g+\widetilde{g}
  \quad\text{for some}\quad 
  \widetilde{g}\in \mathcal D^{\mathrm{ch}}_{\mu,0}.
\end{equation*}

\begin{lemma}[Remainder interaction]
\label{lem:minimal-remainder-interaction}
  Let $0<\mu<1$, $\gamma\in\mathbb R$, and
  $a=c\chi r^\mu$.  If $g\in\mathcal D^{\mathrm{ch}}_{\mu,0}$,
  then
  \begin{equation}
    \label{eq:interaction-remainder}
    G_\gamma(a+g)-G_\gamma(a)-g
    \in
    \mathcal D^{\mathrm{ch}}_{\mu,0}.
  \end{equation}
  The left hand side depends continuously on $(c,g)$ with
  values in $\mathcal D^{\mathrm{ch}}_{\mu,0}$, where this space
  has the
  $L_\mu$ graph norm.
\end{lemma}

\begin{proof}
  Put $\Phi(z)=G_\gamma(z)-z$ and regard $\mathbb C$ as
  $\mathbb R^2$.  On bounded subsets,
  \begin{equation}
    \label{eq:Phi-derivative-bounds}
    |\Phi(z)|
    \lesssim
    |z|^2,
    \qquad
    |D\Phi(z)|
    \lesssim
    |z|,
    \qquad
    |D^2\Phi(z)|
    \lesssim
    1
  \end{equation}
  almost everywhere.  If $v=Ug\in H^2_0$, integration of
  $v''$ twice gives
  \begin{equation}
    \label{eq:minimal-pointwise-bounds}
    |g(r)|
    +
    r|g'(r)|
    \lesssim
    r\|v''\|_{L^2(0,r)}.
  \end{equation}
  In particular, $g=o(r)$ and $g'=o(1)$ at zero.

  We first prove a bound uniform in the angular directions.
  Write
  \begin{equation}
    \label{eq:Phi-radial-factor}
    \Phi(z)
    =
    z h(|z|),
    \qquad
    h(s)
    =
    e^{-i\gamma s}-1.
  \end{equation}
  If $n=z/|z|$ and $z\neq0$, then for
  $p,q\in\mathbb R^2$,
  \begin{equation}
    \label{eq:Phi-explicit-Hessian}
      D^2\Phi(z)[p,q]
      ={}
      h'(|z|)
      \bigl[
        (n\cdot q)p+(n\cdot p)q
      \bigr]
      +
      h''(|z|)
      (n\cdot p)(n\cdot q)z
      +
      \frac{h'(|z|)}{|z|}
      \bigl[
        p\cdot q-(n\cdot p)(n\cdot q)
      \bigr]z.
  \end{equation}
  This formula shows that the Hessian has a bounded but
  direction dependent limit at zero.
  The bounds for the derivatives of $h$, together with
  $|h(s)|\lesssim s$, give
  \begin{equation}
    \label{eq:Phi-third-derivative}
    |D^3\Phi(z)|
    \leq
    \frac{C_{\gamma,R}}{|z|},
    \qquad
    0<|z|\leq R.
  \end{equation}
  Indeed, this follows directly from
  $|D^k|z||\lesssim |z|^{1-k}$ for $k=1,2,3$.
  If $|w|\leq|z|/2$, integrate
  \eqref{eq:Phi-third-derivative} on the segment from $z$ to
  $z+w$.  If $|w|>|z|/2$, use the uniform bound for
  $D^2\Phi$.  In both cases,
  \begin{equation}
    \label{eq:Phi-relative-Hessian}
    |D^2\Phi(z+w)-D^2\Phi(z)|
    \leq
    C_{\gamma,R}
    \frac{|w|}{|z|},
    \qquad
    z\neq0.
  \end{equation}
  Thus the estimate is uniform in the angular directions of
  $z$ and $w$.

  We now prove that $q=U(\Phi(a+g)-\Phi(a))\in H^{2}_{0}$.
  We can write
  \begin{equation}
    \label{eq:interaction-second-derivative}
    q''
    =
    \frac{\mu^2-1/4}{r^2}q
    -
    U L_\mu\bigl(\Phi(a+g)-\Phi(a)\bigr)
  \end{equation}
  and
  \begin{equation}
    \label{eq:radial-composition-identity}
    L_\mu\Phi(h)
    =
    D\Phi(h)L_\mu h
    -
    D^2\Phi(h)[h',h']
    +
    \frac{\mu^2}{r^2}
    \bigl(
      \Phi(h)-D\Phi(h)h
    \bigr).
  \end{equation}
  On an interval where $\chi=1$, applying 
  \eqref{eq:radial-composition-identity} to 
  $\Phi(a+g)$, $\Phi(a)$, subtracting, and using the
  previous estimates for $D^{j}\Phi(z)$
  yields
  \begin{equation}
    \label{eq:interaction-local-bound}
      \bigl|
        L_\mu\bigl(\Phi(a+g)-\Phi(a)\bigr)
      \bigr|
      \lesssim{}
      (|a|+|g|)|L_\mu g|
      +
      |a'||g'|
      +
      |g'|^2
      +
      \frac{|g|}{|a|}|a'|^2
      +
      \frac{(|a|+|g|)|g|}{r^2}.
  \end{equation}
  When $c=0$, the terms containing $a$ are omitted.  When
  $c\neq0$, the estimate for the difference of the second
  derivatives in \eqref{eq:radial-composition-identity}
  follows from \eqref{eq:Phi-relative-Hessian}:
  \begin{equation}
    \label{eq:Phi-second-difference}
    |D^2\Phi(a+g)-D^2\Phi(a)|
    \lesssim
    \frac{|g|}{|a|}.
  \end{equation}
  Since $a=cr^\mu$, every term on the right of
  \eqref{eq:interaction-local-bound} belongs to
  $L^2((0,1),r\,dr)$.  The terms containing $L_\mu g$ are
  integrable because $a$ and $g$ are bounded.  Every remaining
  singular term is bounded by a constant times
  $r^{\mu-1}$, which is square integrable with measure $r\,dr$
  for every $\mu>0$.  We use the one dimensional Rellich
  inequality
  \begin{equation*}
    \left\|\frac{Ug}{r^2}\right\|_2
    \leq
    C\|(Ug)''\|_2,
    \qquad Ug\in H^2_0(\mathbb R_+),
  \end{equation*}
  so the right-hand side below is finite in the minimal graph norm;
  the most singular term then satisfies the direct bound
  \begin{equation}
    \label{eq:relative-Hessian-integrability}
    \left\|
      \frac{|g|}{|a|}|a'|^2
    \right\|_{L^2(r\,dr)}
    \leq
    C_{\mu}|c|
    \left\|
      \frac{Ug}{r^2}
    \right\|_2.
  \end{equation}
  Finally,
  \eqref{eq:Phi-derivative-bounds} and
  \eqref{eq:minimal-pointwise-bounds} give
  \begin{equation}
    \label{eq:interaction-rellich-bound}
    \frac{q}{r^{2}}=
    \frac{
      U\bigl(\Phi(a+g)-\Phi(a)\bigr)
    }{r^2}
    \in
    L^2(0,1).
  \end{equation}
  Away from zero, the standard one dimensional $H^2$
  composition estimate applies.
  Recalling \eqref{eq:interaction-second-derivative},
  we conclude that $q\in H^2$.  Condition
  \eqref{eq:interaction-rellich-bound} forces
  $q(0)=q'(0)=0$, so $q\in H^2_0$.  Approximation by smooth
  minimal domain functions proves membership for general $g$.

  We now prove joint continuity at $c=0$.
  Put
  \begin{equation}
    \label{eq:interaction-map}
    \mathcal I(c,g)
    =
    \Phi(a+g)-\Phi(a)
  \end{equation}
  and consider $c_n\to0$ and
  $g_n\to g$ in $\mathcal D^{\mathrm{ch}}_{\mu,0}$.  The graph
  norm
  equivalence with $H^2_0$ and
  \eqref{eq:minimal-pointwise-bounds} imply
  \begin{equation}
    \label{eq:minimal-L4-convergence}
    g_n'
    \longrightarrow
    g'
    \quad\hbox{in }L^4((0,\infty),r\,dr).
  \end{equation}
  Near zero this follows directly from
  \eqref{eq:minimal-pointwise-bounds}; away from zero it is the
  standard one dimensional $H^2$ estimate.

  Let $a_n=c_n\chi r^\mu$.  We first compare
  $\mathcal I(c_n,g_n)$ with
  $\mathcal I(0,g_n)=\Phi(g_n)$.  Set
  \begin{equation*}
    Q(z)
    =
    \Phi(z)-D\Phi(z)z
    =
    i\gamma |z|e^{-i\gamma|z|}z.
  \end{equation*}
  Then \eqref{eq:Phi-radial-factor} gives, on the bounded range
  of the profiles,
  \begin{equation}
    \label{eq:Q-cross-bound}
    |\Phi(z+w)-\Phi(z)-\Phi(w)|
    +
    |Q(z+w)-Q(z)-Q(w)|
    \leq
    C_{\gamma,R}|z||w|.
  \end{equation}
  In the region where $\chi=1$,
  the difference of their $L_\mu$ images can be written as
  \begin{equation}\label{eq:diff-img}
    \bigl[D\Phi(a_n+g_n)-D\Phi(g_n)\bigr]L_\mu g_n 
    -T_{a,n}-T_{g,n}-T_{a,g,n}
    +\frac{\mu^2}{r^2}
    \bigl[
    Q(a_n+g_n)-Q(a_n)-Q(g_n)
    \bigr],
  \end{equation}
  where the potentially discontinuous Hessian terms are
  \begin{equation}
    \label{eq:continuity-Hessian-terms}
    \begin{split}
      T_{a,n}
      &=
      \bigl[
        D^2\Phi(a_n+g_n)-D^2\Phi(a_n)
      \bigr]
      [a_n',a_n'],
      \\
      T_{g,n}
      &=
      \bigl[
        D^2\Phi(a_n+g_n)-D^2\Phi(g_n)
      \bigr]
      [g_n',g_n'],
      \\
      T_{a,g,n}
      &=
      2D^2\Phi(a_n+g_n)[a_n',g_n'].
    \end{split}
  \end{equation}
  Estimate \eqref{eq:Phi-relative-Hessian} gives, where $\chi=1$,
  \begin{equation*}
    \frac{|g_n|}{|a_n|}|a_n'|^2
    =
    \mu^2|c_n||g_n|r^{\mu-2}
  \end{equation*}
  Thus, writing $v_n=Ug_n$ and using
  $r^{2\mu-4}\le r^{-4}$ on $(0,1)$,
  \begin{equation*}
    \left\|
      \frac{|g_n|}{|a_n|}|a_n'|^2
    \right\|_{L^2(r\,dr)}^2
    \lesssim
    |c_n|^2\int_0^1|v_n(r)|^2r^{2\mu-4}\,dr
    \leq
    |c_n|^2\left\|\frac{v_n}{r^2}\right\|_2^2.
  \end{equation*}
  The last norms are uniformly bounded by the one dimensional
  Rellich inequality on $H^2_0(\mathbb R_+)$; on the cutoff
  transition annulus the same conclusion follows from
  $a_n'=O(c_n)$.  Consequently,
  \begin{equation}
    \label{eq:T-a-continuity}
    \|T_{a,n}\|_{L^2(r\,dr)}
    \leq
    C|c_n|
    \left\|
      \frac{Ug_n}{r^2}
    \right\|_2
    \longrightarrow0.
  \end{equation}
  For $T_{g,n}$, put
  $A_n=D^2\Phi(a_n+g_n)-D^2\Phi(g_n)$.  The decomposition
  \begin{equation*}
    T_{g,n}
    =
    A_n[g_n'-g',g_n']
    +A_n[g',g_n'-g']
    +A_n[g',g']
  \end{equation*}
  shows that the first two terms tend to zero by
  \eqref{eq:minimal-L4-convergence} and the uniform Hessian
  bound.  After passing to a subsequence, $A_n\to0$ almost
  everywhere on $\{g\neq0\}$; on $\{g=0\}$ one has $g'=0$
  almost everywhere, so the last term tends to zero by dominated
  convergence.  Hence
  \begin{equation}
    \label{eq:T-g-continuity}
    \|T_{g,n}\|_{L^2(r\,dr)}
    \longrightarrow0.
  \end{equation}
  The mixed Hessian term $T_{a,g,n}$ is bounded by
  $C\|a_n'g_n'\|_{L^2(r\,dr)}=O(|c_n|)$: near zero use
  $a_n'=O(c_nr^{\mu-1})$ and the uniform pointwise bound for
  $g_n'$, while on the cutoff transition annulus $a_n'=O(c_n)$.
  The term containing $D\Phi$ in \eqref{eq:diff-img} satisfies
  \begin{equation*}
    \left\|
      [D\Phi(a_n+g_n)-D\Phi(g_n)]L_\mu g_n
    \right\|_{L^2(r\,dr)}
    \lesssim
    \|a_n\|_\infty\|L_\mu g_n\|_{L^2(r\,dr)}
    \longrightarrow0
  \end{equation*}
  by the local Lipschitz bound for $D\Phi$. 
  The last term in \eqref{eq:diff-img} 
  tends to zero by the same weighted
  estimate as \eqref{eq:relative-Hessian-integrability}, since
  \eqref{eq:Q-cross-bound} bounds its numerator by
  $C|a_n||g_n|$.
  Thus $\|\mathcal I(c_n,g_n)-\mathcal I(0,g_n)\|_{L^2(r\,dr)}
  \to0$. 
  We conclude that
  \begin{equation}
    \label{eq:interaction-continuity-zero-trace}
    \mathcal I(c_n,g_n)-\mathcal I(0,g_n)
    \longrightarrow0
    \quad\hbox{in the }L_\mu\hbox{ graph norm}.
  \end{equation}
  For $\mathcal I(0,g_n)\to\mathcal I(0,g)$, the Hessian
  difference is decomposed into terms controlled by
  \eqref{eq:minimal-L4-convergence} and
  \begin{equation*}
    [D^2\Phi(g_n)-D^2\Phi(g)][g',g'].
  \end{equation*}
  The latter is handled by the same zero set argument: $D^2\Phi$
  is continuous where $g\neq0$, while $g'=0$ almost everywhere
  on $\{g=0\}$.  The terms containing $D\Phi$ and $Q$ converge
  by their continuity and the weighted minimal-domain bounds.
  If $c_n\to c\neq0$, then $|a_n(r)|\ge c_0r^\mu$ near the pole
  for some $c_0>0$; the same estimates apply without the small
  factor $|c_n|$, and convergence follows from the zero set
  argument and the convergence of the trace coefficients.  This
  proves joint continuity everywhere.
\end{proof}

\begin{theorem}[One channel domain]
\label{the:sharp-channel-domain}
  Let $0<\mu<1$ and $\gamma\in\mathbb R\setminus\{0\}$.
  Decompose $f\in\dom(L_\mu)$ as in
  \eqref{eq:channel-trace-decomposition}.  Then
  \begin{equation}
    \label{eq:sharp-channel-domain-classification}
    G_\gamma(f)\in\dom(L_\mu)
    \quad\Longleftrightarrow\quad
    \mu>\frac12
    \ \text{or}\
    c=0.
  \end{equation}
  On the set selected by the right hand side, $G_\gamma$
  preserves the coefficient $c$ and is continuous in the
  graph norm.  
  For each such $f$,
  \begin{equation}
    \label{eq:channel-modifier-graph-growth}
      \|G_\gamma(f)\|_{\mathrm{graph}(L_\mu)}
      :=
      \|G_\gamma(f)\|_2
      +
      \|L_\mu G_\gamma(f)\|_2
      \leq
      C_f(1+|\gamma|)^2,
      \qquad
      \gamma\in\mathbb R.
  \end{equation}
\end{theorem}

\begin{proof}
  Write $a=c\chi r^\mu$.  By
  Lemma~\ref{lem:minimal-remainder-interaction},
  \begin{equation}
    \label{eq:modifier-domain-reduction}
    G_\gamma(f)
    =
    G_\gamma(a)
    +
    g
    +
    h,
    \qquad
    h\in\mathcal D^{\mathrm{ch}}_{\mu,0}.
  \end{equation}
  It remains to test the trace term.  Near zero, if
  $c\neq0$ and $\kappa=\gamma|c|$, then
  \begin{equation}
    \label{eq:general-trace-modifier-image}
    L_\mu
    \left(
      cr^\mu e^{-i\kappa r^\mu}
    \right)
    =
    c\mu^2r^{2\mu-2}
    e^{-i\kappa r^\mu}
    \left(
      3i\kappa+\kappa^2r^\mu
    \right).
  \end{equation}
  Its squared radial norm is finite at zero exactly when
  \begin{equation}
    \label{eq:trace-integrability-threshold}
    \int_0^1r^{4\mu-3}\,dr
    <
    \infty,
    \qquad\text{that is,}\qquad
    \mu>\frac12.
  \end{equation}
  For this range,
  \begin{equation}
    \label{eq:trace-nonlinear-correction}
    U\bigl(G_\gamma(a)-a\bigr)
    =
    cr^{\mu+1/2}
    \left(
      e^{-i\kappa r^\mu}-1
    \right)
  \end{equation}
  belongs to $H^2_0$ near zero, and the cutoff terms are
  harmless.  Thus
  $G_\gamma(a)-a\in\mathcal D^{\mathrm{ch}}_{\mu,0}$.
  For $0<\mu\leq1/2$, equation
  \eqref{eq:general-trace-modifier-image} excludes membership
  whenever $c\neq0$.  If $c=0$, the interaction lemma with
  $a=0$ gives
  $G_\gamma(g)-g\in\mathcal D^{\mathrm{ch}}_{\mu,0}$.

  Formula \eqref{eq:trace-nonlinear-correction} has zero
  Friedrichs trace, so the coefficient $c$ is unchanged.
  Continuity follows from the topological decomposition
  \eqref{eq:Friedrichs-channel-domain}, the interaction lemma,
  and the explicit trace calculation.  The same composition
  identity
  \eqref{eq:radial-composition-identity} gives
  $|DG_\gamma|\lesssim1+|\gamma||f|$ and
  $|D^2G_\gamma|\lesssim
  |\gamma|+\gamma^2|f|$.  The already proved trace
  integrability gives $(f')^2\in L^2(r\,dr)$ on the sharp set;
  together with the minimal domain estimates, this proves
  \eqref{eq:channel-modifier-graph-growth}.
\end{proof}

For completeness, we show that continuity of $G_{\gamma}$ can
not be upgraded to local Lipschitz continuity. This failure
is unrelated to the pole, and is caused by the lack of a 
continuous second derivative of $z\mapsto ze^{-i\gamma|z|}$ 
at zero.
Thus graph continuity is the best general statement.  A
Lipschitz theorem requires either a weaker
topology or a restriction preventing moving transversal
zeros.

\begin{proposition}[Lipschitz failure at zeros]
\label{prop:graph-Lipschitz-failure}
  For $\gamma\neq0$, the map $G_\gamma$ is not locally
  Lipschitz in the $L_\mu$ graph norm on any graph
  neighborhood of zero contained in the admissible set of
  Theorem~\ref{the:sharp-channel-domain}.
\end{proposition}

\begin{proof}
  Fix $r_0>0$ and a compact interval
  $K\Subset(0,\infty)$ containing $r_0$ in its interior.
  Standard interior elliptic estimates give
  \begin{equation}
    \label{eq:fixed-support-graph-equivalence}
    C_K^{-1}\|h\|_{H^2}
    \leq
    \|h\|_2+\|L_\mu h\|_2
    \leq
    C_K\|h\|_{H^2}
  \end{equation}
  for every $h\in H^2$ supported in $K$.  The weighted and
  unweighted norms are equivalent on this fixed interval.
  (For the lower bound, write $h''$ in terms of $L_\mu h$,
  $h'$, and $h$, and absorb $h'$ with the standard
  $H^2$ interpolation inequality).

  Put $F_\gamma(s)=se^{-i\gamma|s|}$ for real $s$.
  For $s\neq0$, direct differentiation gives
  \begin{equation}
    \label{eq:real-modifier-second-derivative}
    F_\gamma''(s)
    =
    \begin{cases}
      e^{-i\gamma s}(-2i\gamma-\gamma^2s),
      &s>0,\\
      e^{i\gamma s}(2i\gamma-\gamma^2s),
      &s<0.
    \end{cases}
  \end{equation}
  The second derivative has a jump of size $4|\gamma|$ at
  zero.  Hence there is an $\varepsilon_\gamma>0$ such that
  \begin{equation}
    \label{eq:scalar-jump-lower-bound}
    |F_\gamma''(s_+)-F_\gamma''(s_-)|
    \geq
    2|\gamma|
  \end{equation}
  whenever
  $0<s_+<\varepsilon_\gamma$ and
  $-\varepsilon_\gamma<s_-<0$.

  Choose a real $\eta\in C_c^\infty(K)$ equal to one on a
  fixed neighborhood of $r_0$.  For small $\delta>0$, set
  \begin{equation}
    \label{eq:Lipschitz-counterexample}
    f_\delta(r)
    =
    \delta\eta(r)(r-r_0),
    \qquad
    \widetilde f_\delta(r)
    =
    f_\delta(r)+\delta^4\eta(r).
  \end{equation}
  Both functions are smooth and supported away from the pole,
  hence belong to the admissible set for every $0<\mu<1$.
  They tend to zero in the graph norm, and
  \begin{equation}
    \label{eq:Lipschitz-input-distance}
    \|\widetilde f_\delta-f_\delta\|_{
      \mathrm{graph}(L_\mu)}
    =
    \delta^4\|\eta\|_{\mathrm{graph}(L_\mu)}.
  \end{equation}

  Where $\eta=1$, the only zeros of 
  $f_{\delta}$, $\widetilde{f}_{\delta}$ are $r_0$ and
  $r_0-\delta^3$ respectively.
  Consider the middle half of the interval
  between them,
  \begin{equation}
    \label{eq:zero-separation-interval}
    J_\delta
    =
    \left[
      r_0-\frac34\delta^3,
      r_0-\frac14\delta^3
    \right].
  \end{equation}
  On $J_\delta$,
  \begin{equation}
    \label{eq:opposite-sign-ranges}
    -\frac34\delta^4
    \leq
    f_\delta
    \leq
    -\frac14\delta^4,
    \qquad
    \frac14\delta^4
    \leq
    \widetilde f_\delta
    \leq
    \frac34\delta^4.
  \end{equation}
  Moreover, both first derivatives equal $\delta$ and both
  second derivatives vanish there.  If
  \begin{equation*}
    H_\delta
    =
    G_\gamma(\widetilde f_\delta)
    -
    G_\gamma(f_\delta),
  \end{equation*}
  the standard one dimensional chain rule gives
  \begin{equation}
    \label{eq:output-second-derivative}
    H_\delta''
    =
    \delta^2
    \left[
      F_\gamma''(\widetilde f_\delta)
      -
      F_\gamma''(f_\delta)
    \right]
    \qquad\hbox{on }J_\delta.
  \end{equation}
  For sufficiently small $\delta$,
  \eqref{eq:scalar-jump-lower-bound} therefore implies
  \begin{equation}
    \label{eq:output-H2-lower-bound}
    \|H_\delta''\|_{L^2(J_\delta,r\,dr)}
    \geq
    c_{\gamma,r_0}\delta^{7/2}.
  \end{equation}
  Indeed, the pointwise lower bound is
  $2|\gamma|\delta^2$, while
  $|J_\delta|=\delta^3/2$ and $r$ is comparable to $r_0$.
  Since $H_\delta$ is supported in $K$,
  \eqref{eq:fixed-support-graph-equivalence} gives
  \begin{equation}
    \label{eq:Lipschitz-output-distance}
    \|H_\delta\|_{\mathrm{graph}(L_\mu)}
    \geq
    c_{\gamma,\mu,K}\delta^{7/2}.
  \end{equation}
  Dividing \eqref{eq:Lipschitz-output-distance} by
  \eqref{eq:Lipschitz-input-distance} gives a lower bound
  $c\delta^{-1/2}$, which tends to infinity.
\end{proof}

\subsection{Full angular classification}

Fix $0<\alpha<1$ and denote the full minimal domain by
\begin{equation}
  \label{eq:full-minimal-domain}
  \mathcal D^{\mathrm{AB}}_{\alpha,0}
  =
  \operatorname{cl}_{\mathrm{graph}(H_\alpha)}
  C_c^\infty(\mathbb R^2\setminus\{0\}),
\end{equation}
where the graph norm is
$\|\psi\|_2+\|H_\alpha\psi\|_2$.

The following lemma packages standard linear information, 
see~\cite{DG21} for the radial classification 
and~\cite{Fer24,FFFP15} for the corresponding Aharonov--Bohm 
realization and the Friedrichs form characterization.
We include the formulation and proof to
fix the global graph space, trace representatives, and boundary decay
used below.

\begin{lemma}[Minimal domain and traces]
\label{lem:full-minimal-domain}
  Fix a radial cutoff 
  $\chi=\chi_{\mathrm{tr}}\in C_c^\infty([0,\infty))$ 
  equal to one near zero.
  If $0<\alpha<1$, then
  \begin{equation}
    \label{eq:minimal-domain-equals-Y}
    \mathcal D^{\mathrm{AB}}_{\alpha,0}
    =
    \mathcal Y
  \end{equation}
  defined in
  \eqref{eq:pole-compatible-space}--\eqref{eq:pole-compatible-norm},
  with equivalent graph and $\mathcal Y$ norms.  Moreover,
  \begin{equation}
    \label{eq:full-domain-decomposition}
    \dom(H_\alpha)
    =
    \mathcal Y
    \mathbin{\dotplus}
    \mathbb C\chi r^\alpha
    \mathbin{\dotplus}
    \mathbb C\chi r^{1-\alpha}e^{-i\theta}.
  \end{equation}
  The sum is topological in the graph norm, that is, 
  the projection onto
  $\mathcal Y$ and both trace coefficient maps are continuous.
  In particular,
  \begin{equation}
    \label{eq:full-domain-L-infinity}
    \dom(H_\alpha)
    \hookrightarrow
    L^\infty(\mathbb R^2).
  \end{equation}
  Every $\psi\in\mathcal Y$ also satisfies
  \begin{equation}
    \label{eq:minimal-angular-boundary-decay}
    \lim_{r\downarrow0}
    \left\|
      \frac{\psi(r,\cdot)}r
    \right\|_{L^2(\mathbb S^1)}
    =
    0.
  \end{equation}
  Since $\mathcal Y\subset H^2(\mathbb R^2)$ and
  $\psi/|x|^2\in L^2$, every $\psi\in\mathcal Y$ has a continuous
  representative with $\psi(0)=0$ and
  \begin{equation}\label{eq:LinfY}
    \lim_{r\downarrow0}
    \|\psi(r,\cdot)\|_{L^\infty(\mathbb S^1)}
    =
    0.
  \end{equation}
\end{lemma}

\begin{proof}
  Expand a core function 
  $\psi=\sum_{m\in \mathbb{Z}}\psi_{m}(r)e^{im \theta}\in
    C_{c}^{\infty}(\mathbb{R}^{2}\setminus\{0\})$
  in angular modes and put
  \begin{equation}
    \label{eq:q-m-definition}
    q_m(r)
    =
    r^{1/2}\psi_m(r),
    \qquad
    B_{\mu_m}
    =
    -\partial_r^2
    +
    \frac{\mu_m^2-1/4}{r^2}.
  \end{equation}
  Since $\alpha$ is nonintegral, $\mu_m=|m+\alpha|\neq1$ 
  for every $m\in\mathbb Z$.  Lemma~4.2 and Proposition~5.2 
  of Derezi\'nski--Georgescu~\cite{DG21}, in the regimes
  $\mu_m<1$ and $\mu_m>1$, respectively, give
  \begin{equation}
    \label{eq:Bessel-Rellich}
    \left\|
      \frac{q_m}{r^2}
    \right\|_2
    \leq
    \frac1{|1-\mu_m^2|}
    \|B_{\mu_m}q_m\|_2.
  \end{equation}
  In particular in the second regime, write
  $r^{-2}q_m=Z_{\mu_m}B_{\mu_m}q_m$ 
  and apply Proposition~5.2 of~\cite{DG21}
  (note that they denote the Green kernel of $B_{m}$
  by $G_{m}$, which collides with our notation for the
  nonlinear modifier); since for $\tau\in\mathbb R$
  \begin{equation}
    \label{eq:parabola-distance}
    \left|
      \mu_m^2-(1+i\tau)^2
    \right|^2
    =
    (1-\mu_m^2)^2
    +
    2(1+\mu_m^2)\tau^2
    +
    \tau^4,
  \end{equation}
  we see that the distance from $\mu_m^2$ to
  $(1+i\mathbb R)^2$ is exactly $|1-\mu_m^2|$.

  Write $\langle m\rangle=(1+m^2)^{1/2}$.
  Recalling that $\mu_m^2=(m+\alpha)^2$, we see that the constant
  \begin{equation}
    \label{eq:angular-gap-constant}
    c_\alpha
    =
    \inf_{m\in\mathbb Z}
    \frac{
      |1-(m+\alpha)^2|
    }{
      1+m^2
    }
  \end{equation}
  is positive (no numerator vanishes and the quotient
  tends to one as $|m|\to\infty$).  Thus,
  \begin{equation}
    \label{eq:uniform-angular-gap}
    |1-\mu_m^2|
    \geq
    c_\alpha\langle m\rangle^2,
  \end{equation}
  while
  $|\mu_m^2-1/4|\leq C_\alpha\langle m\rangle^2$.
  Estimate \eqref{eq:Bessel-Rellich} and the identity
  \begin{equation*}
    q_m''
    =
    \frac{\mu_m^2-1/4}{r^2}q_m
    -
    B_{\mu_m}q_m
  \end{equation*}
  imply
  \begin{equation}
    \label{eq:channel-minimal-estimate}
    \|q_m''\|_2
    +
    \langle m\rangle^2
    \left\|
      \frac{q_m}{r^2}
    \right\|_2
    \leq
    C_\alpha
    \|B_{\mu_m}q_m\|_2.
  \end{equation}
  Integration by parts also gives
  \begin{equation}
    \label{eq:mixed-Rellich-identity}
    \int_0^\infty\frac{|q_m'|^2}{r^2}\,dr
    =
    -\mathop{\rm Re}
    \int_0^\infty
    q_m''\frac{\overline q_m}{r^2}\,dr
    +
    3\int_0^\infty\frac{|q_m|^2}{r^4}\,dr
  \end{equation}
  since the boundary terms vanish on the punctured core.  By
  Cauchy--Schwarz and
  \eqref{eq:channel-minimal-estimate},
  \begin{equation}
    \label{eq:mixed-Rellich-estimate}
      \langle m\rangle^2
      \left\|
        \frac{q_m'}r
      \right\|_2^2
      \leq{}
      \langle m\rangle^2
      \|q_m''\|_2
      \left\|
        \frac{q_m}{r^2}
      \right\|_2
      +
      3\langle m\rangle^2
      \left\|
        \frac{q_m}{r^2}
      \right\|_2^2
      \leq
      C_\alpha
      \|B_{\mu_m}q_m\|_2^2.
  \end{equation}
  Therefore
  \begin{equation}
    \label{eq:full-channel-Rellich}
    \|q_m''\|_2
    +
    \langle m\rangle
    \left\|
      \frac{q_m'}r
    \right\|_2
    +
    \langle m\rangle^2
    \left\|
      \frac{q_m}{r^2}
    \right\|_2
    \leq
    C_\alpha
    \|B_{\mu_m}q_m\|_2.
  \end{equation}

  We now convert the one dimensional estimates for the 
  radial coefficients $q_m$ into the weighted 
  Cartesian norms that define $\mathcal Y$. Put
  \begin{equation*}
    u_m(r,\theta)
    =
    f_m(r)e^{im\theta},
    \qquad
    f_m(r)
    =
    r^{-1/2}q_m(r).
  \end{equation*}
  In the orthonormal polar frame, the Cartesian Hessian
  satisfies the exact pointwise identity
  \begin{equation}
    \label{eq:polar-Hessian-identity}
      \sum_{j,k=1}^{2}|\partial_{j}\partial_{k}u_{m}|^{2}
      ={}
      |f_m''|^2
      +
      2m^2
      \left|
        \frac{f_m'}r-\frac{f_m}{r^2}
      \right|^2
      +
      \left|
        \frac{f_m'}r
        -
        \frac{m^2f_m}{r^2}
      \right|^2.
  \end{equation}
  Direct differentiation gives
  \begin{equation}
    \label{eq:q-to-polar-derivatives}
    \begin{aligned}
      f_m''
      &=
      r^{-1/2}
      \left(
        q_m''
        -
        \frac{q_m'}r
        +
        \frac{3q_m}{4r^2}
      \right),
      \\
      \frac{f_m'}r-\frac{f_m}{r^2}
      &=
      r^{-3/2}
      \left(
        q_m'
        -
        \frac{3q_m}{2r}
      \right),
      \\
      \frac{f_m'}r-\frac{m^2f_m}{r^2}
      &=
      r^{-3/2}
      \left(
        q_m'
        -
        \frac{(m^2+1/2)q_m}{r}
      \right).
    \end{aligned}
  \end{equation}
  Moreover,
  \begin{equation}
    \label{eq:polar-weighted-first-derivatives}
    \left|
      \frac{\nabla u_m}{r}
    \right|^2
    =
    \frac{|f_m'|^2}{r^2}
    +
    \frac{m^2|f_m|^2}{r^4},
    \qquad
    \left|
      \frac{u_m}{r^2}
    \right|^2
    =
    \frac{|f_m|^2}{r^4}.
  \end{equation}
  Equations \eqref{eq:polar-Hessian-identity}--%
  \eqref{eq:polar-weighted-first-derivatives}, integrated with
  measure $r\,dr\,d\theta$, yield
  \begin{equation}
    \label{eq:polar-Y-channel-bound}
      \|D^2u_m\|_2^2
      +
      \left\|
        \frac{\nabla u_m}{r}
      \right\|_2^2
      +
      \left\|
        \frac{u_m}{r^2}
      \right\|_2^2
      \leq
      C
      \left[
        \|q_m''\|_2^2
        +
        \langle m\rangle^2
        \left\|
          \frac{q_m'}r
        \right\|_2^2
        +
        \langle m\rangle^4
        \left\|
          \frac{q_m}{r^2}
        \right\|_2^2
      \right].
  \end{equation}
  The first derivatives are controlled by the
  $L^2$ norm and the Cartesian Hessian.  Parseval's identity
  and
  \begin{equation}
    \label{eq:channel-graph-Parseval}
    \|H_\alpha\psi\|_2^2
    =
    2\pi
    \sum_{m\in\mathbb Z}
    \|B_{\mu_m}q_m\|_2^2
  \end{equation}
  now allow us to sum the squares of
  \eqref{eq:full-channel-Rellich}.  Together with
  \eqref{eq:polar-Y-channel-bound}, this proves
  \begin{equation}
    \label{eq:Y-from-graph}
    \|\psi\|_{\mathcal Y}
    \leq
    C_\alpha
    \bigl(
      \|\psi\|_2+\|H_\alpha\psi\|_2
    \bigr)
  \end{equation}
  on the punctured core.  Conversely,
  \eqref{eq:Halpha-expanded} gives
  \begin{equation}
    \label{eq:graph-from-Y}
    \|H_\alpha\psi\|_2
    \leq
    \|\Delta\psi\|_2
    +
    2|\alpha|
    \left\|
      \frac{\nabla\psi}{r}
    \right\|_2
    +
    \alpha^2
    \left\|
      \frac{\psi}{r^2}
    \right\|_2
    \leq
    C_\alpha\|\psi\|_{\mathcal Y}.
  \end{equation}
  Taking completions proves
  \eqref{eq:minimal-domain-equals-Y}.
  For $\psi\in\mathcal Y$, estimate
  \eqref{eq:minimal-pointwise-bounds}, applied to every angular
  coefficient, gives
  \begin{equation}
    \label{eq:angular-boundary-decay-estimate}
    \left\|
      \frac{\psi(r,\cdot)}r
    \right\|_{L^2(\mathbb S^1)}^2
    \leq
    C
    \sum_{m\in\mathbb Z}
    \int_0^r|q_m''(s)|^2\,ds.
  \end{equation}
  The sum of the integrals over $(0,\infty)$ is finite by
  \eqref{eq:full-channel-Rellich} and
  \eqref{eq:channel-graph-Parseval}.  Absolute continuity of
  the integral proves
  \eqref{eq:minimal-angular-boundary-decay}.

  For $0<\alpha<1$, only the modes $m=0,-1$ have Bessel order
  below one.  The channel decomposition
  \eqref{eq:Friedrichs-channel-domain}, summed over the other
  minimal channels, gives
  \eqref{eq:full-domain-decomposition}.  Finally,
  $\mathcal Y\subset H^2\subset L^\infty$, while both cutoff
  traces are bounded, hence \eqref{eq:full-domain-L-infinity} 
  follows.
\end{proof}

Put for $\alpha\not\in \mathbb{Z}$
\begin{equation}
  \label{eq:covariant-gradient}
  \nabla_\alpha
  =
  \nabla+iA_\alpha.
\end{equation}
The Friedrichs form domain is
\begin{equation}
  \label{eq:Friedrichs-form-domain}
  \mathcal Q_\alpha
  =
  \left\{
    v\in H^1(\mathbb R^2):
    A_\alpha v\in L^2(\mathbb R^2)
  \right\},
\end{equation}
equipped with its form norm
$\|v\|_{\mathcal Q_\alpha}^{2}:=
  \|v\|_{L^{2}}^{2}+\|\nabla_{\alpha}v\|_{L^{2}}^{2}\simeq
  \|v\|_{H^{1}}^{2}+\||x|^{-1}v\|_{L^{2}}^{2}$,
see~\cite[Proposition~3.8]{Fer24}.
Recall the trace decomposition
\eqref{eq:full-domain-decomposition} and
the sharp class \eqref{eq:full-trace-candidate},
which is endowed with the graph norm of
$H_{\alpha}$.

\begin{lemma}[$L^4$ control]
\label{lem:covariant-L4}
  If $\phi\in\mathcal D_\alpha^\sharp$, then
  \begin{equation}
    \label{eq:covariant-L4}
    \nabla_\alpha\phi
    \in
    L^4(\mathbb R^2).
  \end{equation}
  If $\phi_n\to\phi$ in 
  $\mathcal D_\alpha^\sharp$, then $\phi_n\to\phi$ in
  $L^\infty$ and
  $\nabla_\alpha\phi_n\to\nabla_\alpha\phi$ in $L^4$.
\end{lemma}

\begin{proof}
  For $\psi\in\mathcal Y$, the two estimates
  \begin{equation}
    \label{eq:Y-fourth-power-estimates}
    \|\nabla\psi\|_4^2
    \lesssim
    \|\psi\|_\infty\|D^2\psi\|_2,
    \qquad
    \left\|
      \frac{\psi}{r}
    \right\|_4^2
    \leq
    \|\psi\|_\infty
    \left\|
      \frac{\psi}{r^2}
    \right\|_2
  \end{equation}
  give $\nabla_\alpha\psi\in L^4$.  Any trace allowed in
  \eqref{eq:full-trace-candidate} has order $\mu>1/2$.
  Its $\nabla_{\alpha}$ gradient is $O(r^{\mu-1})$ at zero and hence
  belongs to $L^4$.  The continuity statement follows from the
  topological decomposition
  \eqref{eq:full-domain-decomposition}, the norm equivalence in
  Lemma~\ref{lem:full-minimal-domain}, and the same estimates
  applied to differences.
\end{proof}

\begin{proof}[Proof of Theorem~\ref{the:sharp-full-domain}]
  We first prove sufficiency, thus we assume
  $\phi\in \mathcal{D}^{\sharp}_{\alpha}$.  Write
  $\nabla_{\alpha,j}=\partial_j+iA_{\alpha,j}$.
  On a simply connected set
  $\mathcal O\Subset\mathbb R^2\setminus\{0\}$, choose a real
  function $\chi$ such that $A_\alpha=\nabla\chi$
  (a suitable branch of $\arg z$).  Then
  \begin{equation}
    \label{eq:local-gauge-derivatives}
    \partial_j(e^{i\chi}\phi)
    =
    e^{i\chi}\nabla_{\alpha,j}\phi,
    \qquad
    \partial_j^2(e^{i\chi}\phi)
    =
    e^{i\chi}\nabla_{\alpha,j}^2\phi.
  \end{equation}
  Gauge equivariance
  $G_\gamma(e^{is}z)=e^{is}G_\gamma(z)$ implies
  \begin{equation}
    \label{eq:equivariant-differentials}
    \begin{aligned}
      DG_\gamma(e^{is}z)[e^{is}v]
      &=
      e^{is}DG_\gamma(z)[v],
      \\
      D^2G_\gamma(e^{is}z)
      [e^{is}v,e^{is}w]
      &=
      e^{is}D^2G_\gamma(z)[v,w].
    \end{aligned}
  \end{equation}
  Apply the standard second order Sobolev chain rule to
  $G_\gamma(e^{i\chi}\phi)$ and use
  \eqref{eq:local-gauge-derivatives}--%
  \eqref{eq:equivariant-differentials}.  This gives, on
  $\mathcal O$,
  \begin{equation}
    \label{eq:covariant-first-chain}
    \nabla_{\alpha,j}G_\gamma(\phi)
    =
    DG_\gamma(\phi)
    \nabla_{\alpha,j}\phi
  \end{equation}
  and
  \begin{equation}
    \label{eq:covariant-second-chain}
    H_\alpha G_\gamma(\phi)
    =
    DG_\gamma(\phi)H_\alpha\phi
    -
    \sum_{j=1}^2
    D^2G_\gamma(\phi)
    [
      \nabla_{\alpha,j}\phi,
      \nabla_{\alpha,j}\phi
    ].
  \end{equation}
  Here local elliptic regularity gives
  $\phi\in H^2_{\mathrm{loc}}(\mathcal O)$.
  The map $G_\gamma$ belongs to
  $W^{2,\infty}_{\mathrm{loc}}$, so the identities hold almost
  everywhere and in distributions.
  Note that \eqref{eq:covariant-second-chain} makes sense 
  a.e.~and the singularity at $\phi=0$ of $G_{\gamma}(\phi)$ 
  does not cause a problem: indeed, on the zero set of
  $\phi$, the quadratic Hessian term is independent of the
  representative of $D^2G_\gamma(0)$ because
  $\nabla_{\alpha}\phi=\nabla\phi=0$ almost everywhere there.

  The derivative bounds
  \eqref{eq:G-first-derivative} and
  \eqref{eq:G-second-derivative}, together with
  Lemma~\ref{lem:covariant-L4}, show that the right hand side
  of \eqref{eq:covariant-second-chain} belongs to $L^2$ and
  give
  \begin{equation}
    \label{eq:full-composition-graph-bound}
      \|H_\alpha G_\gamma(\phi)\|_2
      \lesssim{}
      (1+|\gamma|\|\phi\|_\infty)
      \|H_\alpha\phi\|_2
      +
      \bigl(
        |\gamma|
        +
        \gamma^2\|\phi\|_\infty
      \bigr)
      \|\nabla_\alpha\phi\|_4^2.
  \end{equation}
  Thus if we denote the right hand side of
  \eqref{eq:covariant-second-chain} by
  $\mathcal R_\gamma(\phi)$ we have
  $\mathcal R_\gamma(\phi)\in L^2$.

  We must still check that the calculation creates no
  distribution at $x=0$.  Recall the Friedrichs form domain
  \eqref{eq:Friedrichs-form-domain}.
  Since $\phi\in\dom(H_\alpha)\subset\mathcal Q_\alpha$,
  the standard first order Sobolev chain rule gives
  \begin{equation}
    \label{eq:composition-form-bound}
    \|\nabla G_\gamma(\phi)\|_2
    \leq
    C(1+|\gamma|\|\phi\|_\infty)
    \|\nabla\phi\|_2.
  \end{equation}
  Moreover,
  $|G_\gamma(\phi)|=|\phi|$, so
  $A_\alpha G_\gamma(\phi)\in L^2$.
  Hence $G_\gamma(\phi)\in\mathcal Q_\alpha$.

  Let $\mathfrak q_\alpha$ be the closed Friedrichs form.
  The local identity \eqref{eq:covariant-second-chain} and a
  partition of unity give
  \begin{equation}
    \label{eq:punctured-form-identity}
    \mathfrak q_\alpha
    [
      G_\gamma(\phi),\zeta
    ]
    =
    \langle
      \mathcal R_\gamma(\phi),\zeta
    \rangle
  \end{equation}
  for every
  $\zeta\in C_c^\infty(\mathbb R^2\setminus\{0\})$.
  This punctured space is a form core by the definition of the
  Friedrichs realization.  Both sides of
  \eqref{eq:punctured-form-identity} are continuous in the form
  norm, so the identity extends to every
  $\zeta\in\mathcal Q_\alpha$.  The representation theorem for
  closed forms now gives
  \begin{equation}
    \label{eq:global-covariant-chain}
    G_\gamma(\phi)\in\dom(H_\alpha),
    \qquad
    H_\alpha G_\gamma(\phi)
    =
    \mathcal R_\gamma(\phi).
  \end{equation}
  This proves the global form of
  \eqref{eq:covariant-second-chain}, ruling out a point mass at
  the pole, and yields
  \eqref{eq:full-modifier-graph-growth}.

  We prove that $G_{\gamma}(\phi)$ has the same Friedrichs 
  traces as $\phi$.
  Since $\phi\in \mathcal D_\alpha^\sharp$, its trace
  term is either 0 or is $O(|x|^{\mu})$ at 0 for some
  $\mu>1/2$, while its minimal component vanishes at 0 and is in
  $\mathcal Y\subset H^2(\mathbb{R}^{2})\subset C^{0,\beta}$
  for all $\beta<1$, so it is $O(|x|^{\beta})$.
  Thus $\phi=o(|x|^{\frac 12})$ at 0, and
  \begin{equation}
    \label{eq:nonlinear-correction-order}
    |G_\gamma(\phi)-\phi|
    \leq
    |\gamma||\phi|^2
    =
    o(r).
  \end{equation}
  Hence the nonlinear correction has zero Friedrichs traces, and
  both trace coefficients are preserved.

  For graph continuity, let $\phi_n\to\phi$ in the graph norm
  inside $\mathcal D_\alpha^\sharp$.
  Lemmas~\ref{lem:full-minimal-domain}
  and~\ref{lem:covariant-L4} imply
  \begin{equation}
    \label{eq:sharp-convergence}
    \phi_n\longrightarrow\phi
    \quad\hbox{in }L^\infty,
    \qquad
    \nabla_\alpha\phi_n
    \longrightarrow
    \nabla_\alpha\phi
    \quad\hbox{in }L^4.
  \end{equation}
  We now apply \eqref{eq:covariant-second-chain}, where
  the only subtle point is the discontinuity of
  $D^2G_\gamma$ at zero.  On $\{\phi\neq0\}$ use pointwise
  convergence.  On $\{\phi=0\}$ use the Sobolev space property
  that
  $\nabla\phi=0$ almost everywhere.  The uniform bound on
  $D^2G_\gamma$ and dominated convergence then prove graph
  continuity.  The same argument with $\phi$ fixed and
  $\gamma_n\to\gamma$ proves graph continuity in $\gamma$.

  Conversely, assume $\phi\in\operatorname{dom}(H_\alpha)$
  with
  $G_{\gamma}(\phi)\in \operatorname{dom}(H_\alpha)$.
  Suppose first that
  $\alpha\neq1/2$, put
  \begin{equation}
    \label{eq:lower-trace-order}
    \nu=\nu_{\alpha}
    =
    \min\{\alpha,1-\alpha\}
    <
    \frac12,
  \end{equation}
  and let $m_{\nu}=0$ if $\alpha<\frac 12$,
  $m_{\nu}=-1$ if $\alpha>\frac 12$.
  Then \eqref{eq:full-domain-decomposition}, 
  \eqref{eq:minimal-angular-boundary-decay}
  give a coefficient $c\in\mathbb C$ such that
  \begin{equation}
    \label{eq:lower-trace-remainder-size}
    \|\phi(r,\theta)-cr^{\nu}e^{im_{\nu}\theta}\|
      _{L^2(\mathbb S^1_{\theta})}
    =
    o(r^\nu).
  \end{equation}
  We must prove that $c=0$.
  By \eqref{eq:LinfY} and the explicit trace terms in
  \eqref{eq:full-domain-decomposition}, we also have
  \begin{equation*}
    \|\phi(r,\cdot)\|_{L^\infty(\mathbb S^1)}=o(1).
  \end{equation*}
  Therefore,
  by \eqref{eq:lower-trace-remainder-size},
  \begin{equation}
    \label{eq:lower-trace-quadratic-projection}
    \left\|
      |\phi|\phi
      -
      |c|c r^{2\nu}e^{im_\nu\theta}
    \right\|_{L^1(\mathbb S^1)}
    \leq
    C\bigl(\|\phi\|_{L^2(\mathbb S^1)}+\|c r^{\nu}e^{im_\nu\theta}\|
      _{L^2(\mathbb S^1)}\bigr)
    \|\phi-c r^{\nu}e^{im_\nu\theta}\|_{L^2(\mathbb S^1)}
    =
    o(r^{2\nu})
  \end{equation}
  and
  \begin{equation}
    \label{eq:lower-trace-cubic-remainder}
    \||\phi|^3\|_{L^1(\mathbb S^1)}
    \leq
    \|\phi\|_\infty
    \|\phi\|_{L^2(\mathbb S^1)}^2
    =
    o(r^{2\nu}).
  \end{equation}
  Taking the $m_\nu$ projection in
  $G_\gamma(\phi)-\phi=-i\gamma|\phi|\phi+O(|\phi|^3)$
  gives
  \begin{equation}
    \label{eq:lower-trace-generated-power}
    p(r):=
    \bigl(
      G_\gamma(\phi)-\phi
    \bigr)_{m_\nu}
    =
    -i\gamma|c|c\,r^{2\nu}
    +
    o(r^{2\nu}).
  \end{equation}
  Recall now that both $G_\gamma(\phi)$ and $\phi$,
  hence their difference, belong to the operator domain.
  In the $m_{\nu}$ channel, the Friedrichs trace of elements
  of the operator domain is of the form $c'r^{\nu}$;
  comparing with \eqref{eq:lower-trace-generated-power},
  this shows that $c'=0$ and hence
  $p(r)\in \mathcal D^{\mathrm{ch}}_{\nu,0}$.
  We conclude that $p(r)=o(r)$.  Since $2\nu<1$,
  this implies $c=0$.

  It remains to consider the case $\alpha=1/2$.
  Then the decomposition
  \eqref{eq:full-domain-decomposition} takes the form, near zero,
  \begin{equation}
    \label{eq:half-flux-expansion}
    \phi(r,\theta)
    =
    r^{1/2}A(\theta)
    +
    \psi(r,\theta),
    \qquad
    A(\theta)
    =
    c_0+c_{-1}e^{-i\theta},
    \quad
    \psi\in\mathcal Y.
  \end{equation}
  We must prove that $c_{0}=c_{-1}=0$, or equivalently
  $A \equiv0$. Put
  \begin{equation}
    \label{eq:half-flux-rescaled-profile}
    v_r(\theta)
    =
    r^{-1/2}\phi(r,\theta)
    =
    A(\theta)+\varepsilon_r(\theta),
    \qquad
    \varepsilon_r
    =
    r^{-1/2}\psi(r,\cdot).
  \end{equation}
  By \eqref{eq:minimal-angular-boundary-decay} and the
  embedding $\mathcal Y\subset L^\infty$,
  \begin{equation}
    \label{eq:half-flux-epsilon-bounds}
    \|\varepsilon_r\|_2
    =
    o(r^{1/2}),
    \qquad
    \|\varepsilon_r\|_\infty
    \leq
    r^{-1/2}\|\psi\|_\infty.
  \end{equation}
  Hence
  \begin{equation}
    \label{eq:half-flux-L4-convergence}
    \|\varepsilon_r\|_4^2
    \leq
    \|\varepsilon_r\|_\infty
    \|\varepsilon_r\|_2
    =
    o(1).
  \end{equation}
  Thus $v_r\to A$ in $L^4(\mathbb S^1)$.  The same
  interpolation also gives
  \begin{equation}
    \label{eq:half-flux-cubic-remainder}
    r^{1/2}\|\varepsilon_r\|_6^3
    \leq
    r^{1/2}
    \|\varepsilon_r\|_\infty^2
    \|\varepsilon_r\|_2
    =
    o(1).
  \end{equation}
  Thus
  $r^{1/2}\|v_r\|_6^3\to0$.
  Now, the estimate
  \begin{equation}
    \label{eq:phase-Taylor-remainder}
    |e^{-i\gamma s}-1+i\gamma s|
    \leq
    C_\gamma s^2,
    \qquad
    s\geq0,
  \end{equation}
  implies
  \begin{equation*}
    \frac{G_\gamma(\phi)-\phi}{r}
    =
    r^{-1/2}v_r
    \left(
    e^{-i\gamma r^{1/2}|v_r|}-1
    \right)
    =
    -i\gamma |v_r|v_r
    +
    O\!\left(r^{1/2}|v_r|^3\right).
  \end{equation*}
  Therefore
  \begin{equation*}
  \begin{split}
    \left\|
      \frac{
        G_\gamma(\phi(r,\cdot))-\phi(r,\cdot)
      }r
      +
      i\gamma|A|A
    \right\|_{L^2(\mathbb S^1)}
    &\leq
    \left\|
    \frac{G_\gamma(\phi)-\phi}{r}
    +i\gamma |A|A
    \right\|_{L^2}
    \\
    &\leq
    |\gamma|
    \bigl\||v_r|v_r-|A|A\bigr\|_{L^2}
    +
    C_\gamma r^{1/2}\|v_r\|_{L^6}^3.
    \end{split}
  \end{equation*}
  The first term tends to zero because $v_r\to A$ in $L^4$,
  the second by \eqref{eq:half-flux-cubic-remainder} as already
  noticed, hence
  \begin{equation}
    \label{eq:half-flux-generated-power}
    \lim_{r\downarrow0}
    \left\|
      \frac{
        G_\gamma(\phi(r,\cdot))-\phi(r,\cdot)
      }r
      +
      i\gamma|A|A
    \right\|_{L^2(\mathbb S^1)}
    =
    0.
  \end{equation}

  By assumption
  $h=G_\gamma(\phi)-\phi\in \dom(H_{1/2})$.  Convergence
  \eqref{eq:half-flux-generated-power} gives
  $r^{-1/2}h(r,\cdot)\to0$ in $L^2(\mathbb S^1)$.
  On the other hand, by \eqref{eq:full-domain-decomposition}
  we can write
  $r^{-1/2}h(r,\theta) = d_0+d_{-1}e^{-i\theta} +
    r^{-1/2}\psi(r,\theta)$
  and by
  \eqref{eq:minimal-angular-boundary-decay} 
  \begin{equation*}
    \|r^{-1/2}\psi(r,\cdot)\|_{L^2}
    =
    r^{1/2}
    \left\|\frac{\psi(r,\cdot)}{r}\right\|_{L^2}
    \longrightarrow0.
  \end{equation*}
  Therefore $d_{0}=d_{-1}=0$, that is, $h\in \mathcal{Y}$.
  But then \eqref{eq:minimal-angular-boundary-decay}
  implies $r^{-1}h(r,\cdot)\to0$ in $L^2(\mathbb S^1)$,
  and by comparison with \eqref{eq:half-flux-generated-power} 
  we conclude $|A|A=0$.
\end{proof}

\begin{corollary}[Admissibility]
\label{cor:sharp-full-admissibility}
  Let $0<\alpha<1$ and
  $\phi\in\mathcal D_\alpha^\sharp$.  Then $\phi$ is admissible
  in the sense of Definition~\ref{def:admissible-profile}.
\end{corollary}

\begin{proof}
  Apply \eqref{eq:full-modifier-graph-growth} with
  $\gamma=\lambda\log(t)/2$.  The graph continuity argument in
  the proof of Theorem~\ref{the:sharp-full-domain}, with
  $\gamma_n\to\gamma$ and $\phi$ fixed, gives the continuity
  required in Definition~\ref{def:admissible-profile}.
\end{proof}

We have seen that if $\phi\not\in\mathcal{D}^{\sharp}_{\alpha}$,
then $G_{\gamma}(\phi)$ may fall outside the Friedrichs domain.
In this case, as a substitute,
we can use a suitable extension of the magnetic form
associated to $H_{\alpha}$.
If $\nabla_\alpha v\in L^{2p}_{\mathrm{loc}}$ for some $p>1$,
put $s=(2p)'<2$ and define
\begin{equation}
  \label{eq:extended-covariant-distribution}
  \langle H_\alpha v,\zeta\rangle
  =
  \int_{\mathbb R^2}
  \nabla_\alpha v\mathbin{\cdot}
  \overline{\nabla_\alpha\zeta}
  \,dx,
  \qquad
  \zeta\in C_c^\infty(\mathbb R^2).
\end{equation}
Indeed, $\nabla_\alpha\zeta\in L^s$ locally because
$|A_\alpha(x)|\lesssim |x|^{-1}$ and $s<2$.  Thus
\eqref{eq:extended-covariant-distribution} defines a
distribution.  It agrees with the closed Friedrichs form
pairing whenever both $v$ and the test function belong to
$\mathcal Q_\alpha$, the proper form domain.

\begin{lemma}[Image of the Friedrichs domain]
\label{lem:full-domain-weak-image}
  Let $0<\alpha<1$, let
  $\phi\in\dom(H_\alpha)$, and assume
  \begin{equation}
    \label{eq:full-domain-p-range}
    1<p<\frac1{1-\nu_\alpha},
    \qquad
    \nu_{\alpha}=\min\{\alpha,1-\alpha\}.
  \end{equation}
  For every $\gamma\in\mathbb R$,
  $G_\gamma(\phi)$ belongs to the Friedrichs form domain.
  Moreover, there are
  $F_{2,\gamma}\in L^2$ and $F_{p,\gamma}\in L^p$ such that
  \begin{equation}
    \label{eq:weak-full-domain-chain}
    H_\alpha G_\gamma(\phi)
    =
    F_{2,\gamma}+F_{p,\gamma}
    \quad\hbox{in }\mathcal D'(\mathbb R^2),
  \end{equation}
  and
  \begin{equation}
    \label{eq:weak-full-domain-bound}
    \|F_{2,\gamma}\|_2
    +
    \|F_{p,\gamma}\|_p
    \leq
    C_{\phi,p}(1+|\gamma|)^2.
  \end{equation}
  For every compactly supported cutoff $\eta$, one also has
  \begin{equation}
    \label{eq:weak-modifier-local-gradient}
    \|\eta\nabla_\alpha G_\gamma(\phi)\|_{2p}
    \leq
    C_{\phi,p,\eta}(1+|\gamma|).
  \end{equation}
  For the fixed cutoff $\eta$, the components in
  \eqref{eq:weak-chain-split} depend continuously on $\gamma$
  in $L^p$ and $L^2$, respectively; moreover,
  $\gamma\mapsto G_\gamma(\phi)$ is continuous in the form
  topology.
  In addition,
  \begin{equation}
    \label{eq:full-domain-form-growth}
    \|H_\alpha^{1/2}G_\gamma(\phi)\|_2
    \leq
    C_\phi(1+|\gamma|).
  \end{equation}
\end{lemma}

\begin{proof}
  Let $\eta\in C_c^\infty(\mathbb R^2)$ be equal to one near 0.
  Write $\phi$ as in
  \eqref{eq:full-domain-decomposition}.  The proof of
  Lemma~\ref{lem:covariant-L4} gives
  $\nabla_\alpha\psi\in L^4$ for the minimal part
  $\psi\in \mathcal{Y}$.  A trace
  vector of order $\mu$ satisfies
  \begin{equation}
    \label{eq:trace-gradient-size}
    \left|
      \nabla_\alpha
      \bigl(\chi r^\mu e^{im\theta}\bigr)
    \right|
    \leq
    C r^{\mu-1}
    \quad\hbox{near zero}.
  \end{equation}
  Since both trace orders are at least $\nu_\alpha$,
  \eqref{eq:trace-gradient-size} and
  \eqref{eq:full-domain-p-range} imply
  \begin{equation}
    \label{eq:full-domain-gradient-integrability}
    \eta\nabla_\alpha\phi\in L^{2p},
    \qquad
    (1-\eta)\nabla_\alpha\phi\in L^4.
  \end{equation}

  Define as a distribution on $\mathbb{R}^{2}\setminus\{0\}$
  \begin{equation}
    \label{eq:weak-chain-right-hand-side}
    \mathcal R_\gamma(\phi)
    =
    DG_\gamma(\phi)H_\alpha\phi
    -
    \sum_{j=1}^2
    D^2G_\gamma(\phi)
    [
      \nabla_{\alpha,j}\phi,
      \nabla_{\alpha,j}\phi
    ].
  \end{equation}
  The local chain rule
  \eqref{eq:covariant-second-chain}, the derivative bounds
  \eqref{eq:G-first-derivative}--%
  \eqref{eq:G-second-derivative}, and
  \eqref{eq:full-domain-gradient-integrability} show that
  \begin{equation}
    \label{eq:weak-chain-split}
    F_{p,\gamma}
    :=
    \eta\mathcal R_\gamma(\phi)
    \in L^p,
    \qquad
    F_{2,\gamma}
    :=
    (1-\eta)\mathcal R_\gamma(\phi)
    \in L^2.
  \end{equation}
  Here $H_\alpha\phi\in L^p$ on bounded sets because
  $p<2$.  These estimates give
  \eqref{eq:weak-full-domain-bound}.

  Gauge equivariance and the first order chain rule give
  \begin{equation}
    \label{eq:full-domain-covariant-first-chain}
    \nabla_{\alpha,j}G_\gamma(\phi)
    =
    DG_\gamma(\phi)\nabla_{\alpha,j}\phi.
  \end{equation}
  Integrability \eqref{eq:full-domain-gradient-integrability} and
  \eqref{eq:full-domain-covariant-first-chain}, the bound for
  $DG_\gamma$, and the $L^4$ estimate away from the pole give
  \eqref{eq:weak-modifier-local-gradient}.
  The standard Sobolev chain rule gives
  $G_\gamma(\phi)\in H^1$, while
  $|G_\gamma(\phi)|=|\phi|$ gives
  $A_\alpha G_\gamma(\phi)\in L^2$.  Thus
  $G_\gamma(\phi)$ belongs to the Friedrichs form domain, and
  \eqref{eq:full-domain-form-growth} follows.

  Let $\gamma_n\to\gamma$.  Since $\phi$ is bounded, the matrices
  $DG_{\gamma_n}(\phi(x))$ converge pointwise to
  $DG_\gamma(\phi(x))$ and remain uniformly bounded.  By
  \eqref{eq:full-domain-covariant-first-chain} and dominated
  convergence,
  \begin{equation*}
    \nabla_\alpha G_{\gamma_n}(\phi)
    \longrightarrow
    \nabla_\alpha G_\gamma(\phi)
    \quad\hbox{in }L^2.
  \end{equation*}
  Pointwise convergence of the modifiers, together with
  $|G_{\gamma_n}(\phi)|=|\phi|$, gives
  \begin{equation*}
    G_{\gamma_n}(\phi)
    \longrightarrow
    G_\gamma(\phi)
    \quad\hbox{in }L^2.
  \end{equation*}
  Hence $G_{\gamma_n}(\phi)\to G_\gamma(\phi)$ in the
  Friedrichs form topology.
  The same dominated convergence argument applies to the first
  term in \eqref{eq:weak-chain-right-hand-side}.  For the Hessian
  term, use pointwise convergence on $\{\phi\neq0\}$.  On
  $\{\phi=0\}$, one has
  $\nabla_\alpha\phi=0$ almost everywhere away from the pole, so
  the possible ambiguity in $D^2G_\gamma(0)$ is harmless.
  The domination needed near the pole and away from it follows
  from \eqref{eq:full-domain-gradient-integrability}.  Therefore
  the two components in \eqref{eq:weak-chain-split} are continuous
  in $L^p$ and $L^2$, respectively.

  We must still prove the second order identity across the
  pole, that is, \eqref{eq:weak-full-domain-chain} holds
  and no point mass at 0 is overlooked
  (i.e.~\eqref{eq:weak-chain-right-hand-side} holds
  in distributions on $\mathbb{R}^{2}$).
  Put $s=(2p)'<2$.  Fix
  $\zeta\in C_c^\infty(\mathbb R^2)$ and choose a radial cutoff
  $\chi_\varepsilon$ which vanishes for
  $r\leq\varepsilon$, equals one for
  $r\geq2\varepsilon$, and satisfies
  $|\nabla\chi_\varepsilon|\leq C\varepsilon^{-1}$.
  Set $\zeta_\varepsilon=\chi_\varepsilon\zeta$.
  The punctured identity gives
  \begin{equation}
    \label{eq:weak-chain-punctured-test}
    \mathfrak q_\alpha
    [G_\gamma(\phi),\zeta_\varepsilon]
    =
    \langle
      F_{2,\gamma}+F_{p,\gamma},
      \zeta_\varepsilon
    \rangle.
  \end{equation}
  On $B_{2\varepsilon}$, the size of $A_\alpha$ and the
  cutoff bounds imply the following estimate.  The ordinary
  product rule gives three terms: the cutoff derivative, the
  smooth gradient of $\zeta$, and the magnetic term.  Their
  $L^s$ norms are, respectively,
  $O(\varepsilon^{2/s-1})$,
  $O(\varepsilon^{2/s})$, and
  $O(\varepsilon^{2/s-1})$.  Hence
  \begin{equation}
    \label{eq:weak-chain-cutoff-bound}
    \|\nabla_\alpha
      (\zeta_\varepsilon-\zeta)\|_{L^s(B_{2\varepsilon})}
    \leq
    C_\zeta\varepsilon^{2/s-1}
    =
    C_\zeta\varepsilon^{1-1/p}.
  \end{equation}
  The exponent in \eqref{eq:weak-chain-cutoff-bound} is
  positive.  For all sufficiently small $\varepsilon$,
  H\"older's inequality gives
  \begin{equation*}
    \left|
      \int_{\mathbb R^2}
      \nabla_\alpha G_\gamma(\phi)
      \mathbin{\cdot}
      \overline{
        \nabla_\alpha(\zeta_\varepsilon-\zeta)
      }
    \right|
    \leq
    \|\eta\nabla_\alpha G_\gamma(\phi)\|_{2p}
    \|\nabla_\alpha
      (\zeta_\varepsilon-\zeta)\|_s
    \longrightarrow0.
  \end{equation*}
  Thus the left hand side of
  \eqref{eq:weak-chain-punctured-test} has a limit.  On the
  right hand side,
  $\zeta_\varepsilon\to\zeta$ in both $L^2$ and $L^{p'}$.
  The limiting left hand side is the covariant distributional
  pairing which defines $H_\alpha G_\gamma(\phi)$; it is finite
  by the $L^{2p}$ and $L^s$ bounds above.  The resulting
  identity holds for every smooth test function and is exactly
  \eqref{eq:weak-full-domain-chain}.  Thus the
  limiting calculation also excludes a distribution supported
  at the pole.
\end{proof}

The preceding weak image and local gradient bounds allow the
pseudoconformal calculation to be extended to the full Friedrichs
domain in distributional form.

\begin{lemma}[Weak pseudoconformal identity]
\label{lem:weak-pseudoconformal}
  Let $I\Subset(0,\infty)$ and let
  \begin{equation*}
    W\in C(I;\mathcal Q_\alpha)\cap C^1(I;L^2).
  \end{equation*}
  Suppose, for some $1<p<2$, that
  \begin{equation}
    \label{eq:weak-pseudoconformal-gradient}
    \nabla_\alpha W
    \in
    L^1_{\mathrm{loc}}
    \bigl(I;L^{2p}_{\mathrm{loc}}\bigr)
  \end{equation}
  and, in $\mathcal D'(I\times\mathbb R^2)$,
  \begin{equation*}
    H_\alpha W(t)
    =
    F(t),
    \qquad
    F\in L^1_{\mathrm{loc}}
    \bigl(I;L^2+L^p\bigr).
  \end{equation*}
  Here, and in the conclusion, the magnetic Laplacian on
  smooth tests is understood through the extended covariant
  pairing \eqref{eq:extended-covariant-distribution}.
  Then, in $\mathcal D'(I\times\mathbb R^2)$,
  \begin{equation}
    \label{eq:weak-pseudoconformal-identity}
    (i\partial_t-H_\alpha)M(t)D(t)W(t)
    =
    M(t)D(t)
    \left(
      i\partial_tW(t)-\frac{F(t)}{4t^2}
    \right).
  \end{equation}
\end{lemma}

\begin{proof}
  Put $U=M D W$ and
  $\mathcal K=M D(i\partial_tW-F/(4t^2))$.
  We first work on the punctured cylinder.  If
  $\Psi\in C_c^\infty
  (I\times(\mathbb R^2\setminus\{0\}))$, then
  \begin{equation}
    \label{eq:weak-pseudoconformal-punctured}
    \langle
      (i\partial_t-H_\alpha)U,\Psi
    \rangle
    =
    \langle\mathcal K,\Psi\rangle.
  \end{equation}
  Indeed, localize to a compact punctured cylinder and
  regularize $W$ in space and time.  The coefficients of
  $H_\alpha$ are smooth there, so the commutators created by
  regularization tend to zero in distributions on smaller
  compact punctured cylinders.  For the regularized functions,
  the change of variables $x=2ty$ and
  \eqref{eq:first-covariant-calculation}--%
  \eqref{eq:time-calculation} apply directly.  Homogeneity
  \eqref{eq:A-homogeneous}, tangentiality
  \eqref{eq:A-tangential}, and $H_\alpha W=F$ give
  \eqref{eq:weak-pseudoconformal-punctured} in the limit.

  We now pass across the pole.  On compact time intervals,
  \eqref{eq:first-covariant-calculation},
  \eqref{eq:weak-pseudoconformal-gradient}, and the local
  Sobolev embedding of $\mathcal Q_\alpha$ imply
  \begin{equation}
    \label{eq:weak-pseudoconformal-U-gradient}
    \nabla_{\alpha,x}U
    \in
    L^1_{\mathrm{loc}}
    \bigl(I;L^{2p}_{\mathrm{loc}}\bigr).
  \end{equation}
  Indeed, after $x=2ty$, the local bound is
  $|\nabla_{\alpha,x}U|\lesssim_I
  |yW|+|\nabla_{\alpha,y}W|$.  The first term is controlled by
  the local Sobolev embedding, and the second by
  \eqref{eq:weak-pseudoconformal-gradient}.
  Let $s=(2p)'$ and fix
  $\Psi\in C_c^\infty(I\times\mathbb R^2)$.  Use the cutoffs
  from the proof of
  Lemma~\ref{lem:full-domain-weak-image}.  Uniformly for $t$
  in the support of $\Psi$, one has
  \begin{equation}
    \label{eq:weak-pseudoconformal-cutoff}
    \|\nabla_{\alpha,x}
      ((\chi_\varepsilon-1)\Psi(t))\|_s
    \leq
    C_\Psi\varepsilon^{1-1/p}.
  \end{equation}
  Let $J\Subset I$ and $B_R=\{|x|<R\}$ such that 
  $J \times B_{R}$ contains the support of $\Psi$.  
  Equations \eqref{eq:weak-pseudoconformal-U-gradient} and
  \eqref{eq:weak-pseudoconformal-cutoff}, after one magnetic
  integration by parts, give
  \begin{equation}
    \label{eq:weak-pseudoconformal-H-limit}
      \left|
        \int_J\int_{B_R}
        \nabla_{\alpha,x}U
        \mathbin{\cdot}
        \overline{
          \nabla_{\alpha,x}
          ((1-\chi_\varepsilon)\Psi)
        }
        \,dx\,dt
      \right|
      \leq
      C_\Psi\varepsilon^{1-1/p}
      \|\nabla_{\alpha,x}U\|_{
        L^1(J;L^{2p}(B_R))}
      \longrightarrow0.
  \end{equation}
  Thus \eqref{eq:weak-pseudoconformal-H-limit} passes the
  magnetic second order term to the limit.
  The time derivative also passes to the limit.  Indeed, strong
  continuity of $M(t)$ and $D(t)$ on $L^2$ makes $U$ bounded
  in $C(J;L^2)$, while
  \begin{equation*}
    \|(1-\chi_\varepsilon)\partial_t\Psi\|_{
      L^1(J;L^2)}
    \longrightarrow0.
  \end{equation*}
  Finally, $\mathcal K$ belongs locally in time to
  $L^2+L^p$.  The dilation factors are uniformly bounded on
  $J$, and $(1-\chi_\varepsilon)\Psi$ tends to zero in
  $L^\infty(J;L^2\cap L^{p'})$.  Hence the right hand side of
  \eqref{eq:weak-pseudoconformal-punctured}, tested against
  $\chi_\varepsilon\Psi$, converges to
  $\langle\mathcal K,\Psi\rangle$.  Taking
  $\varepsilon\downarrow0$ proves
  \eqref{eq:weak-pseudoconformal-identity} for every $\Psi$ and
  excludes a distribution supported at the pole.
\end{proof}

\section{Modified final states, sharpness, and the first
correction}
\label{sec:all-channel}

The final state construction has two branches.  The first uses
the admissible class of
Definition~\ref{def:admissible-profile}; the sharp domain
profiles belong to it by
Corollary~\ref{cor:sharp-full-admissibility}.  The second uses
the weak full domain image of
Lemma~\ref{lem:full-domain-weak-image} to include forbidden
traces through a mixed dual Strichartz estimate.  The two
branches are combined in
Subsection~\ref{sec:modified-wave-operator} to prove
Theorem~\ref{the:main-final-state}.
The next subsection contains a proof of the sharpness
Theorem \ref{the:pure-trace-rate-sharpness}.
At the end we construct the first 
correction to the logarithmic profile.

\begin{theorem}[Modified final states for admissible profiles]
\label{the:all-channel}
  Fix $1/2<b<1$ and $\lambda\in\mathbb R$.  Let
  $\phi\in L^2\cap L^\infty$ be an admissible profile.
  There is an
  $\varepsilon_{b,\alpha,\lambda}>0$ such that, if
  \begin{equation}
    \label{eq:general-smallness}
    \|\phi\|_\infty
    <
    \varepsilon_{b,\alpha,\lambda},
  \end{equation}
  then \eqref{eq:main-NLS} has a global solution satisfying
  \begin{equation}
    \label{eq:all-channel-asymptotic}
    \sup_{\tau\geq T}
    \tau^b
    \bigl(
      \|u-u_{\mathrm{ap}}\|_{
        L^\infty([\tau,\infty);L^2)}
      +
      \|u-u_{\mathrm{ap}}\|_{
        L^4([\tau,\infty);L^4)}
    \bigr)
    <
    \infty
  \end{equation}
  for some $T\geq1$, where $u_{\mathrm{ap}}$ is defined by
  \eqref{eq:all-channel-uap}.
  The solution is unique in the class
  \eqref{eq:all-channel-asymptotic}.
  In particular, this result applies to any 
  $\phi\in \mathcal{D}^{\sharp}_{\alpha}$
  satisfying \eqref{eq:general-smallness}, by
  Corollary \ref{cor:sharp-full-admissibility}.
\end{theorem}

\begin{theorem}[Modified final states for all Friedrichs traces]
\label{the:full-domain-final-states}
  Let $0<\alpha<1$ and
  $1<p<\frac1{1-\nu_\alpha}$, and put
  \begin{equation}
    \label{eq:full-domain-decay-exponent}
    \delta_p
    =
    \frac1{q_p'}
    =
    \frac32-\frac1p.
  \end{equation}
  Fix
  \begin{equation}
    \label{eq:full-domain-b-range}
    \frac12<b<\delta_p
  \end{equation}
  and $\lambda\in\mathbb R$.  There is
  $\varepsilon_{b,\alpha,\lambda}>0$, independent of $p$,
  such that every
  $\phi\in\dom(H_\alpha)$ with
  \begin{equation}
    \label{eq:full-domain-smallness}
    \|\phi\|_\infty
    <
    \varepsilon_{b,\alpha,\lambda}
  \end{equation}
  determines a unique global solution of
  \eqref{eq:main-NLS} in the final state class
  \eqref{eq:all-channel-asymptotic}.
\end{theorem}

Since $p$ may approach $(1-\nu_\alpha)^{-1}$ from below,
Theorem~\ref{the:full-domain-final-states} allows every
\begin{equation}
  \label{eq:full-domain-optimal-b-range}
  \frac12<b<\frac12+\nu_\alpha.
\end{equation}
At half flux this is the entire range $1/2<b<1$ of the
weighted contraction.  Away from half flux, the sharp domain
still gives the stronger range $1/2<b<1$ through
Theorem~\ref{the:all-channel}.

\subsection{The weighted final state space}

For $T\geq1$ and $b>1/2$, define
\begin{equation}
  \label{eq:X-space}
  \|z\|_{X_{b,T}}
  =
  \sup_{\tau\geq T}
  \tau^b
  \bigl(
    \|z\|_{L^\infty([\tau,\infty);L^2)}
    +
    \|z\|_{L^4([\tau,\infty);L^4)}
  \bigr).
\end{equation}
We shall use the next elementary estimate 
(but only for $b>1/2$):

\begin{lemma}[Weighted quadratic estimate]
\label{lem:weighted-quadratic}
  If $b>1/2$ and $z\in X_{b,T}$, then for every
  $\tau\geq T$,
  \begin{equation}
    \label{eq:L2L4-weighted}
    \|z\|_{L^2([\tau,\infty);L^4)}
    \leq
    C_b\tau^{1/4-b}\|z\|_{X_{b,T}}.
  \end{equation}
  Hence,
  \begin{equation}
    \label{eq:z-square-estimate}
    \||z|^2\|_{L^1([\tau,\infty);L^2)}
    \leq
    C_b\tau^{1/2-2b}
    \|z\|_{X_{b,T}}^2.
  \end{equation}
\end{lemma}

\begin{proof}
  Let
  $I_k=[2^k\tau,2^{k+1}\tau]$.  H\"older's inequality
  in time and \eqref{eq:X-space} give
  \begin{equation}
    \label{eq:dyadic-L2L4}
      \|z\|_{L^2(I_k;L^4)}
      \leq
      |I_k|^{1/4}
      \|z\|_{L^4(I_k;L^4)}
      \leq
      (2^k\tau)^{1/4-b}
      \|z\|_{X_{b,T}}.
  \end{equation}
  Squaring and summing in $k$ proves
  \eqref{eq:L2L4-weighted},
  and \eqref{eq:z-square-estimate} follows trivially.
\end{proof}

\subsection{The contraction}

\begin{proof}[Proof of Theorem~\ref{the:all-channel}]
  Let $u_{\mathrm{ap}}$ be \eqref{eq:uap-definition} or, in
  explicit form, \eqref{eq:all-channel-uap}.  We want
  to construct
  \begin{equation}
    \label{eq:uap-plus-z}
    u=u_{\mathrm{ap}}+z.
  \end{equation}
  With $N(w)=|w|w$ and $R(t)$ defined by
  \eqref{eq:residual-definition}, the equation for $z$ is
  \begin{equation}
    \label{eq:z-equation}
    (i\partial_t-H_\alpha)z
    =
    \lambda
    \bigl(
      N(u_{\mathrm{ap}}+z)-N(u_{\mathrm{ap}})
    \bigr)
    -
    R.
  \end{equation}
  The final condition $z(\infty)=0$ leads to reformulate
  \eqref{eq:z-equation} as a fixed point problem
  \begin{equation}
    \label{eq:fixed-point-map}
    z=
    \Phi(z)(t)
    :=
    i\int_t^\infty
    e^{-i(t-s)H_\alpha}
    \bigl[
      \lambda
      \bigl(
        N(u_{\mathrm{ap}}+z)
        -
        N(u_{\mathrm{ap}})
      \bigr)
      -
      R(s)
    \bigr]\,ds.
  \end{equation}

  We use
  \begin{equation}
    \label{eq:N-Lipschitz}
    |N(v+w)-N(v)|
    \leq
    C\bigl(|v||w|+|w|^2\bigr).
  \end{equation}
  The approximate solution satisfies
  \begin{equation}
    \label{eq:uap-L-infinity}
    \|u_{\mathrm{ap}}(s)\|_\infty
    =
    \frac1{2s}\|\phi\|_\infty.
  \end{equation}
  Hence, for $\tau\geq T$,
  \begin{equation}
    \label{eq:linear-long-range-term}
      \|u_{\mathrm{ap}}z\|_{
        L^1([\tau,\infty);L^2)}
      \leq
      \frac{\|\phi\|_\infty}{2}
      \int_\tau^\infty
      s^{-1}\|z(s)\|_2\,ds
      \leq
      \frac{\|\phi\|_\infty}{2b}
      \tau^{-b}\|z\|_{X_{b,T}}.
  \end{equation}
  Lemma~\ref{lem:weighted-quadratic} gives
  \begin{equation}
    \label{eq:quadratic-remainder-term}
    \||z|^2\|_{L^1([\tau,\infty);L^2)}
    \leq
    C_b\tau^{1/2-2b}
    \|z\|_{X_{b,T}}^2.
  \end{equation}

  Finally, once $T$ is sufficiently large,
  Proposition~\ref{prop:approximate-solution} gives
  \begin{equation}
    \label{eq:residual-tail}
      E_T
      :=
      \sup_{\tau\geq T}
      \tau^b
      \|R\|_{L^1([\tau,\infty);L^2)}
      \leq
      C_{\phi,b}
      T^{b-1}(1+\log T)^2.
  \end{equation}
  In particular, $E_T\to0$ as $T\to\infty$.

  Apply the Strichartz estimate \eqref{eq:retarded-Strichartz} 
  on every interval $[\tau,\infty)$,
  multiply by $\tau^{b}$ and then take the supremum in $\tau$.
  Estimates \eqref{eq:linear-long-range-term},
  \eqref{eq:quadratic-remainder-term}, and
  \eqref{eq:residual-tail} imply
  \begin{equation}
    \label{eq:fixed-point-bound}
    \|\Phi(z)\|_{X_{b,T}}
    \leq
    C_\alpha
    \biggl[
      \frac{|\lambda|\|\phi\|_\infty}{b}
      \|z\|_{X_{b,T}}
      +
      |\lambda|T^{1/2-b}
      \|z\|_{X_{b,T}}^2
      +
      E_T
    \biggr].
  \end{equation}
  The difference estimate has the same form:
  \begin{equation}
    \label{eq:fixed-point-difference}
      \|\Phi(z_1)-\Phi(z_2)\|_{X_{b,T}}
      \leq
      C_\alpha|\lambda|
      \biggl[
        \frac{\|\phi\|_\infty}{b}
        +
        T^{1/2-b}
        \|z_1\|_{X_{b,T}}
        +
        T^{1/2-b}
        \|z_2\|_{X_{b,T}}
      \biggr]
      \|z_1-z_2\|_{X_{b,T}}.
  \end{equation}

  Choose $\varepsilon_{b,\alpha,\lambda}$ so that the first
  coefficient in \eqref{eq:fixed-point-difference} is less than
  $1/4$.  If $\lambda=0$, no smallness is needed.  Next choose
  $T$ large.  Since $b>1/2$,
  $T^{1/2-b}\to0$, while $E_T\to0$.  Therefore
  $\Phi$ is a contraction on the ball
  \begin{equation}
    \label{eq:fixed-point-ball}
    \left\{
      z\in X_{b,T}:
      \|z\|_{X_{b,T}}
      \leq
      2C_\alpha E_T
    \right\}.
  \end{equation}
  This gives a unique solution on $[T,\infty)$ satisfying
  \eqref{eq:all-channel-asymptotic}.

  We now identify this tail solution with the global $L^2$
  solution.  The inhomogeneous Strichartz construction makes
  the Duhamel term continuous in $L^2$.  Moreover,
  $z\in L^\infty L^2\cap L^4L^4$ implies
  $z\in L^6_{\mathrm{loc}}L^3$ by interpolation.  Since
  $\phi\in L^2\cap L^\infty\subset L^3$ and
  $\|u_{\mathrm{ap}}(t)\|_3
  =(2t)^{-1/3}\|\phi\|_3$, the same local bound holds for
  $u_{\mathrm{ap}}$.  Thus the tail solution belongs to
  \begin{equation*}
    C([T,\infty);L^2)
    \cap
    L^6_{\mathrm{loc}}([T,\infty);L^3),
  \end{equation*}
  which is the uniqueness class in
  Proposition~\ref{prop:global-L2-theory}.  Applying that
  proposition with the datum $u(T)$ extends the solution to
  all real times and identifies the extension with the tail
  solution on $[T,\infty)$.

  Finally, any solution of \eqref{eq:z-equation} belonging to
  the stated final state class satisfies
  the fixed point equation $z=\Phi(z)$.
  To see this, integrate its remainder
  equation from $t$ to a finite terminal time $S$ and let
  $S\to\infty$; the terminal remainder tends to zero in
  $L^2$, while the residual and nonlinear source converge in
  $L^1L^2$.  As a consequence, by \eqref{eq:fixed-point-difference},
  we obtain uniqueness of solutions of \eqref{eq:z-equation}
  belonging to the final class.
\end{proof}

We now derive the mixed residual estimate needed in the proof of
Theorem~\ref{the:full-domain-final-states}.

\begin{proposition}[Full domain residual]
\label{prop:full-domain-residual}
  Under the hypotheses of
  Theorem~\ref{the:full-domain-final-states}, define $W$ by
  \eqref{eq:W-definition} and $u_{\mathrm{ap}}$ by
  \eqref{eq:uap-definition}.  There are
  $R_2(t)\in L^2$ and $R_p(t)\in L^p$ such that
  \begin{equation}
    \label{eq:full-domain-residual-identity}
    (i\partial_t-H_\alpha)u_{\mathrm{ap}}
    -
    \lambda|u_{\mathrm{ap}}|u_{\mathrm{ap}}
    =
    R_2+R_p
  \end{equation}
  in distributions.  For $t\geq1$,
  \begin{equation}
    \label{eq:full-domain-residual-pointwise}
    \|R_2(t)\|_2
    \leq
    C_{\phi,p}t^{-2}(1+\log t)^2,
    \qquad
    \|R_p(t)\|_p
    \leq
    C_{\phi,p}t^{-3+2/p}(1+\log t)^2.
  \end{equation}
  Hence, for every $\tau\geq1$,
  \begin{equation}
    \label{eq:full-domain-residual-tail}
    \|R_2\|_{L^1([\tau,\infty);L^2)}
    +
    \|R_p\|_{L^{q_p'}([\tau,\infty);L^p)}
    \leq
    C_{\phi,p}
    \tau^{-\delta_p}(1+\log\tau)^2.
  \end{equation}
\end{proposition}

\begin{proof}
  Lemma~\ref{lem:full-domain-weak-image}, with
  $\gamma=\lambda\log(t)/2$, gives
  \begin{equation}
    \label{eq:full-domain-W-image}
    H_\alpha W(t)
    =
    F_2(t)+F_p(t),
    \qquad
    \|F_2(t)\|_2+\|F_p(t)\|_p
    \leq
    C_{\phi,p}(1+\log t)^2.
  \end{equation}
  The form continuity,
  \eqref{eq:weak-modifier-local-gradient}, and
  \eqref{eq:W-ODE} imply that the hypotheses of
  Lemma~\ref{lem:weak-pseudoconformal} are fulfilled.
  Hence the weak pseudoconformal identity holds with
  $H_\alpha W=F_2+F_p$.  The lemma is local in time: apply it
  on every compact subinterval of $(1,\infty)$ to obtain the
  distributional identity on the full interval.

  Equation \eqref{eq:W-ODE} and nonlinear scaling now give
  \eqref{eq:full-domain-residual-identity}, where
  \begin{equation}
    \label{eq:full-domain-residual-components}
    R_j(t)
    =
    -\frac1{4t^2}M(t)D(t)F_j(t),
    \qquad
    j\in\{2,p\}.
  \end{equation}
  Since
  \begin{equation}
    \label{eq:dilation-Lp-scaling}
    \|D(t)f\|_p
    =
    (2t)^{2/p-1}\|f\|_p,
  \end{equation}
  estimates \eqref{eq:full-domain-W-image},
  \eqref{eq:full-domain-residual-components}, and
  \eqref{eq:dilation-Lp-scaling} imply
  \eqref{eq:full-domain-residual-pointwise}.
  Finally, noticing that
  \begin{equation}
    \label{eq:full-domain-time-balance}
    q_p'\left(3-\frac2p\right)
    =
    2,
  \end{equation}
  we see that integration in time gives the $L^{q_p'}L^p$ term 
  in \eqref{eq:full-domain-residual-tail}.  
  On the other hand, the $L^1L^2$ term
  decays like $\tau^{-1}(1+\log\tau)^2$, which is stronger
  because $\delta_p<1$.
\end{proof}

The smallness threshold for $\|\phi\|_\infty$ is independent of $p$.
For a profile with
a nonzero lowest trace, the residual constant $C_{\phi,p}$ and
the admissible starting time generally deteriorate as
$p\uparrow(1-\nu_\alpha)^{-1}$.  Thus approaching the rate
threshold costs time and constants, but not additional
$L^\infty$ smallness.

\begin{proof}[Proof of
  Theorem~\ref{the:full-domain-final-states}]
  Use the same fixed point map \eqref{eq:fixed-point-map}, with
  $R=R_2+R_p$ from
  Proposition~\ref{prop:full-domain-residual}.  The pair dual
  to the $L^{q_p'}L^p$ residual is
  \begin{equation}
    \label{eq:full-domain-dual-pair}
    q_p
    =
    \frac{2p}{2-p},
    \qquad
    r_p
    =
    \frac p{p-1}.
  \end{equation}
  The pair in \eqref{eq:full-domain-dual-pair} is nonendpoint
  admissible.  Therefore
  \eqref{eq:retarded-Strichartz} and
  \eqref{eq:mixed-retarded-Strichartz}, with all norms below
  taken on $[\tau,\infty)$, give
  \begin{equation}
    \label{eq:full-domain-retarded-bound}
    \left\|
      \int_t^\infty
      e^{-i(t-s)H_\alpha}
      \bigl(R_2(s)+R_p(s)\bigr)\,ds
    \right\|_{L^\infty L^2\cap L^4L^4}
    \leq
    C_{\alpha,p}
    \left(
      \|R_2\|_{L^1L^2}
      +
      \|R_p\|_{L^{q_p'}L^p}
    \right).
  \end{equation}
  Set
  \begin{equation}
    \label{eq:full-domain-weighted-error}
    E_{p,T}
    =
    \sup_{\tau\geq T}
    \tau^b
    \left(
      \|R_2\|_{L^1([\tau,\infty);L^2)}
      +
      \|R_p\|_{L^{q_p'}([\tau,\infty);L^p)}
    \right).
  \end{equation}
  By \eqref{eq:full-domain-weighted-error} and
  \eqref{eq:full-domain-residual-tail},
  \begin{equation}
    \label{eq:full-domain-weighted-error-bound}
    E_{p,T}
    \leq
    C_{\phi,p}
    T^{b-\delta_p}(1+\log T)^2
    \longrightarrow
    0.
  \end{equation}

  The nonlinear estimates
  \eqref{eq:linear-long-range-term} and
  \eqref{eq:quadratic-remainder-term} are unchanged.
  Estimates \eqref{eq:full-domain-retarded-bound}--%
  \eqref{eq:full-domain-weighted-error-bound} give
  the fixed point bounds
  \eqref{eq:fixed-point-bound} and
  \eqref{eq:fixed-point-difference}, with $E_T$ replaced by
  $C_{\alpha,p}E_{p,T}$.  Smallness of $\|\phi\|_\infty$
  absorbs the linear long range term.  Its coefficient is the
  same as in \eqref{eq:fixed-point-difference} and is independent
  of $p$, so the smallness threshold is uniform in $p$.
  Since $b>1/2$ and $b<\delta_p$, a sufficiently
  large $T$ controls the quadratic term and the residual.  The
  contraction gives the unique solution in
  \eqref{eq:all-channel-asymptotic}.

  The rest of the proof is similar; in particular,
  when proving uniqueness in the final state class,
  integrate from $t$ to $S$ and let $S\to \infty$:
  then the nonlinear source converges in
  $L^1L^2$, and the two residual components converge in
  $L^1L^2$ and $L^{q_p'}L^p$, respectively.  Estimate
  \eqref{eq:full-domain-retarded-bound} passes to the limit,
  and the same contraction proves uniqueness.
\end{proof}

The preceding theorems specialize to each angular channel.
In this setting the linear flow and the gauge invariant nonlinearity
preserve the channel, and the trace order \(\mu\) directly determines
the admissible decay range. We give the resulting explicit radial
statement.

\begin{corollary}[Single channel final states]
\label{cor:natural-channel-final-states}
  Let $1/2<b<1$, let $m\in\{0,-1\}$, and
  \begin{equation}
    \label{eq:natural-channel-order}
    0 < \mu
    =
    |m+\alpha| < 1.
  \end{equation}
  Suppose $f\in\dom(L_\mu)\cap L^\infty$ has the decomposition
  \eqref{eq:channel-trace-decomposition}, and assume
  \begin{equation}
    \label{eq:natural-channel-trace-condition}
    c=0
    \quad\text{or}\quad
    b<\frac12+\mu.
  \end{equation}
  If $\|f\|_\infty$ is sufficiently small, then
  the modified final state construction applies to
  \begin{equation}
    \label{eq:natural-channel-profile}
    \phi(r,\theta)
    =
    e^{im\theta}f(r).
  \end{equation}
  The resulting solution remains in this angular channel and
  has the approximate final state
  \begin{equation}
    \label{eq:natural-channel-approximate-state}
    u_{\mathrm{ap}}(t,r,\theta)
    =
    \frac{e^{im\theta}}{2it}
    e^{ir^2/(4t)}
    f\left(\frac r{2t}\right)
    \exp\left[
      -\frac{i\lambda}{2}
      \left|
        f\left(\frac r{2t}\right)
      \right|
      \log t
    \right].
  \end{equation}
\end{corollary}

\begin{proof}
  If $c=0$ or $\mu>1/2$,
  Theorem~\ref{the:sharp-channel-domain} and
  \eqref{eq:channel-modifier-graph-growth} show that
  \eqref{eq:natural-channel-profile} is admissible, so
  Theorem~\ref{the:all-channel} applies.  In the remaining case,
  $c\neq0$ and $\mu\leq1/2$.  Then
  $\mu=\nu_\alpha$, and
  Theorem~\ref{the:full-domain-final-states} applies by
  \eqref{eq:natural-channel-trace-condition}.
  The linear flow preserves the channel by
  \eqref{eq:channel-linear}.  The identity
  $|e^{im\theta}f|e^{im\theta}f=e^{im\theta}|f|f$
  treats the nonlinearity.
\end{proof}


\subsection{The modified nonlinear wave operator and the main
theorem}
\label{sec:modified-wave-operator}

Recall $u_+$ and $\mathcal N_\alpha(t)$ from
\eqref{eq:u-plus-definition} and
\eqref{eq:distorted-nonlinear-modifier}.

\begin{proposition}[Form of the asymptotic state]
\label{prop:modified-wave-form}
  Let $\phi\in\dom(H_\alpha)$, and define $W$ by
  \eqref{eq:W-definition}.  Then
  \begin{equation}
    \label{eq:full-domain-linear-form-error}
    \left\|
      M(t)D(t)W(t)
      -
      e^{-itH_\alpha}\mathcal F_\alpha^{-1}W(t)
    \right\|_2
    \leq
    C_\phi t^{-1/2}(1+\log t).
  \end{equation}
  If $\phi$ is admissible, the stronger estimate
  \begin{equation}
    \label{eq:linear-form-error}
    \left\|
      M(t)D(t)W(t)
      -
      e^{-itH_\alpha}\mathcal F_\alpha^{-1}W(t)
    \right\|_2
    \leq
    \frac{C_\phi}{4t}(1+\log t)^2
  \end{equation}
  holds.  Hence every solution given by
  Theorem~\ref{the:all-channel} or
  Theorem~\ref{the:full-domain-final-states} satisfies
  \begin{equation}
    \label{eq:modified-wave-asymptotic}
    \left\|
      u(t)
      -
      e^{-itH_\alpha}
      \mathcal N_\alpha(t)u_+
    \right\|_2
    \longrightarrow0
    \qquad
    (t\to\infty).
  \end{equation}
\end{proposition}

\begin{proof}
  By \eqref{eq:full-factorization},
  \begin{equation}
    \label{eq:factorization-difference}
    e^{-itH_\alpha}
    \mathcal F_\alpha^{-1}W(t)
    -
    M(t)D(t)W(t)
    =
    M(t)D(t)
    \mathcal F_\alpha
    \bigl(M(t)-1\bigr)
    \mathcal F_\alpha^{-1}W(t).
  \end{equation}
  All factors apart from $M(t)-1$ are unitary.  Since
  \begin{equation}
    \label{eq:M-minus-one-form}
    |e^{i|x|^2/(4t)}-1|
    \leq
    C t^{-1/2}|x|,
  \end{equation}
  functional calculus in the dual intertwining identity
  \eqref{eq:distorted-dual-intertwining} gives
  \begin{equation}
    \label{eq:dual-form-intertwining}
    \left\|
      |x|\mathcal F_\alpha^{-1}W(t)
    \right\|_2
    =
    \|H_\alpha^{1/2}W(t)\|_2.
  \end{equation}
  Lemma~\ref{lem:full-domain-weak-image} and
  \eqref{eq:M-minus-one-form} prove
  \eqref{eq:full-domain-linear-form-error}.

  If $\phi$ is admissible, we also use
  \begin{equation}
    \label{eq:M-minus-one}
    |e^{i|x|^2/(4t)}-1|
    \leq
    \frac{|x|^2}{4t},
  \end{equation}
  the dual intertwining identity
  \eqref{eq:distorted-dual-intertwining} gives
  \begin{equation}
    \label{eq:factorization-error-bound}
      \left\|
        \bigl(M(t)-1\bigr)
        \mathcal F_\alpha^{-1}W(t)
      \right\|_2
      \leq
      \frac1{4t}
      \left\|
        |x|^2\mathcal F_\alpha^{-1}W(t)
      \right\|_2
      =
      \frac1{4t}\|H_\alpha W(t)\|_2.
  \end{equation}
  Admissibility proves \eqref{eq:linear-form-error}.
  Combining either estimate with
  \eqref{eq:all-channel-asymptotic} gives
  \eqref{eq:modified-wave-asymptotic}.
\end{proof}

\begin{proof}[Proof of Theorem~\ref{the:main-final-state}]
  Suppose first that $b<1/2+\nu_\alpha$.  Choose and fix
  $p$ such that
  $1/(3/2-b)<p<(1-\nu_\alpha)^{-1}$.  This interval is
  nonempty exactly because $b<1/2+\nu_\alpha$.
  Theorem~\ref{the:full-domain-final-states} applies to every
  $\phi\in\mathcal P_{\alpha,b}=\dom(H_\alpha)$ and gives
  \eqref{eq:main-final-state-bound} with a smallness constant
  depending only on $b$, $\alpha$, and $\lambda$.
  Take $\varepsilon_{b,\alpha,\lambda}$ in the present theorem
  to be this constant.
  If instead $1/2+\nu_\alpha\leq b<1$, then
  $\mathcal P_{\alpha,b}=\mathcal D_\alpha^\sharp$.
  Corollary~\ref{cor:sharp-full-admissibility} and
  Theorem~\ref{the:all-channel} give the unique global solution
  and \eqref{eq:main-final-state-bound}.
  In this range, take $\varepsilon_{b,\alpha,\lambda}$ to be
  the threshold supplied by Theorem~\ref{the:all-channel}.

  In either branch,
  Proposition~\ref{prop:modified-wave-form} gives
  \eqref{eq:main-modified-wave-asymptotic}.  Hence the value
  $u(0)$ depends uniquely on $u_+$ in
  \eqref{eq:wave-operator-domain}, which proves that
  \eqref{eq:wave-operator-map} is well defined.

  To prove injectivity, suppose that two final profiles
  $\phi_1$ and $\phi_2$ produce the same value at time zero.
  Global $L^2$ uniqueness gives the same solution for all
  times, and \eqref{eq:main-final-state-bound} implies
  \begin{equation*}
    \||\phi_{1}|-|\phi_{2}|\|_{2}\leq
    \left\|
      G_{\lambda\log(t)/2}(\phi_1)
      -
      G_{\lambda\log(t)/2}(\phi_2)
    \right\|_2
    \longrightarrow0
  \end{equation*}
  so that
  $|\phi_1|=|\phi_2|$ almost everywhere.  The two modifiers
  then contain the same scalar phase, so their distance equals
  $\|\phi_1-\phi_2\|_2$ for every $t$.  Hence
  $\phi_1=\phi_2$, and the unitarity of $\mathcal F_\alpha$
  gives equality of the final data.
\end{proof}

\subsection{Sharpness of the trace dependent rate}
\label{sec:pure-trace-rate-sharpness}

We now prove
Theorem~\ref{the:pure-trace-rate-sharpness}.  The argument uses
only the exceptional channel and a high frequency Hankel
projection.

Fix $0<\nu<1/2$ and a cutoff $\chi$ as in
\eqref{eq:pure-trace-profile}.  Put
\begin{equation}
  \label{eq:rate-sharpness-model-power}
  g_\nu(r)
  =
  \chi(r)r^{2\nu}.
\end{equation}
The following cutoff Hankel asymptotic is a special case of the
expansions in Wong~\cite{Won76,Won77}; for completeness, we give
an independent proof of the coefficient and the remainder needed
here.

\begin{lemma}[Hankel threshold]
\label{lem:rate-Hankel-threshold}
  As $\rho\to\infty$,
  \begin{equation}
    \label{eq:rate-Hankel-tail}
    (\mathcal H_\nu g_\nu)(\rho)
    =
    \kappa_\nu\rho^{-2\nu-2}
    +
    o(\rho^{-2\nu-2}),
  \end{equation}
  where
  \begin{equation}
    \label{eq:rate-Hankel-coefficient}
    \kappa_\nu
    =
    2^{2\nu+1}
    \frac{\Gamma((3\nu+2)/2)}{\Gamma(-\nu/2)}
    \neq0.
  \end{equation}
\end{lemma}

\begin{proof}
  We separate the singular power from the cutoff by introducing
  the Gaussian model
  \begin{equation*}
    g_0(r)=r^{2\nu}e^{-r^2},
    \qquad
    h(r)=g_\nu(r)-g_0(r).
  \end{equation*}
  The Gaussian--Bessel integral gives
  \begin{equation*}
    (\mathcal H_\nu g_0)(\rho)
    =
    \frac{\Gamma(A)}{2^{\nu+1}\Gamma(B)}\,
    \rho^\nu
    {}_1F_1\!\left(A;B;-\frac{\rho^2}{4}\right),
    \qquad
    A=\frac{3\nu+2}{2},
    \qquad
    B=\nu+1.
  \end{equation*}
  The positive-axis asymptotic expansion of the Kummer function
  (see~\cite[\S13.7(i)]{DLMF}) is
  \begin{equation*}
    {}_1F_1(A;B;-x)
    =
    \frac{\Gamma(B)}{\Gamma(B-A)}x^{-A}
    \bigl(1+O(x^{-1})\bigr)
    +
    O\bigl(e^{-x}x^{A-B}\bigr)
    \qquad (x\to+\infty).
  \end{equation*}
  Since $B-A=-\nu/2$, this yields
  \begin{equation*}
    (\mathcal H_\nu g_0)(\rho)
    =
    2^{2\nu+1}
    \frac{\Gamma((3\nu+2)/2)}{\Gamma(-\nu/2)}
    \rho^{-2\nu-2}
    +
    O(\rho^{-2\nu-4})
    +
    O\bigl(\rho^{2\nu}e^{-\rho^2/4}\bigr).
  \end{equation*}

  It remains to show that the replacement of the Gaussian by
  $\chi$ contributes only a lower-order term.  Since $\chi=1$
  near zero,
  \begin{equation*}
    h(r)=O(r^{2\nu+2}),
    \qquad
    L_\nu h(r)=O(r^{2\nu}),
    \qquad
    L_\nu^2h(r)=O(r^{2\nu-2})
    \quad (r\downarrow0).
  \end{equation*}
  At infinity, $h$ and all its derivatives decay rapidly.  Because
  $\nu>0$, $L_\nu^2h$ belongs to $L^1((0,\infty),r\,dr)$.
  Integrating by parts twice with the Bessel operator gives
  \begin{equation*}
    (\mathcal H_\nu h)(\rho)
    =
    \rho^{-4}
    \int_0^\infty
      J_\nu(\rho r)(L_\nu^2h)(r)r\,dr.
  \end{equation*}
  The boundary terms vanish: at the origin this follows from
  $J_\nu(\rho r)=O(r^\nu)$ and
  $J_\nu'(\rho r)=O(r^{\nu-1})$, together with the three
  displayed bounds, and at infinity from rapid decay.  Since
  $J_\nu$ is bounded,
  \begin{equation*}
    |(\mathcal H_\nu h)(\rho)|
    \leq
    C_\nu\rho^{-4}
    \|L_\nu^2h\|_{L^1(r\,dr)}.
  \end{equation*}
  As $0<\nu<1/2$, this is
  $o(\rho^{-2\nu-2})$.  Combining the Gaussian asymptotic with
  this remainder proves \eqref{eq:rate-Hankel-tail}, and the
  coefficient is exactly \eqref{eq:rate-Hankel-coefficient}.
\end{proof}


The following identity is at the root of the extra
logarithm:
\begin{equation}
  \label{eq:rate-logarithmic-Duhamel-symbol}
    \rho^2
    \int_0^s
    e^{i(s-\sigma)\rho^2}
    \log\frac1\sigma
    \,d\sigma
    =
    \frac{e^{is\rho^2}-1}{i}
    \log\frac1s
    +
    s\rho^2
    \int_0^1
    e^{i(1-\tau)s\rho^2}
    \log\frac1\tau
    \,d\tau
\end{equation}
(to prove it, just set $\sigma=s \tau$ in the first integral).

The Taylor remainder in the nonlinear phase starts at order
$r^{3\nu}$.  The following elementary consequence of the cutoff
Hankel expansion is the only estimate on that remainder needed
below.

\begin{lemma}[Hankel bound for the nonlinear remainder]
\label{lem:rate-remainder-hankel-bound}
  Fix $a\in\mathbb R$.  Let $f=c\chi r^\nu$ and define
  \begin{equation}
    \label{eq:rate-remainder-profile}
    Q_\ell
    =
    f\left(e^{-ia\ell|f|}-1+ia\ell|f|\right),
    \qquad \ell\in\mathbb R.
  \end{equation}
  There are constants $C$ and $M$, independent of $\ell$, such
  that
  \begin{equation}
    \label{eq:rate-remainder-hankel-bound}
    \left|(\mathcal H_\nu Q_\ell)(\rho)\right|
    \leq
    C(1+|\ell|)^M\rho^{-3\nu-2},
    \qquad \rho\geq1.
  \end{equation}
\end{lemma}

\begin{proof}
  Near the origin, Taylor expansion gives
  $Q_\ell(r)=A_\ell r^{3\nu}+O((1+|\ell|)^3r^{4\nu})$,
  where $|A_\ell|\lesssim(1+|\ell|)^2$, with the same
  polynomial control for the derivatives needed in the cutoff
  Hankel expansion.  The part supported away from the origin is
  smooth; repeated integration by parts with $L_\nu$ gives the
  same decay there, with constants polynomial in $|\ell|$.
  The cutoff Hankel expansion, as in the proof of
  Lemma~\ref{lem:rate-Hankel-threshold}, therefore gives
  \eqref{eq:rate-remainder-hankel-bound}.  The estimate is uniform
  in $\ell$ after increasing $M$.
\end{proof}

\begin{lemma}[Pure trace response]
\label{lem:rate-pure-trace-response}
  Fix $0<\nu<1/2$ and let $\chi$ be a radial cutoff near 0.
  Let $c\in\mathbb C\setminus\{0\}$
  and $a\in\mathbb R\setminus\{0\}$, and put
  \begin{equation*}
    f(r)
    =
    c\chi(r)r^\nu,
    \qquad
    W_s(r)
    =
    f(r)
    \exp\left(
      -ia|f(r)|\log\frac1{4s}
    \right).
  \end{equation*}
  For $0<s<1/4$, define the retarded response by the $L^2$ limit
  \begin{equation}
    \label{eq:rate-linear-response}
    Y_{\mathrm{lin}}(s)
    =
    i\int_0^s
    e^{i(s-\sigma)L_\nu}
    L_\nu W_\sigma
    \,d\sigma
    :=
    \lim_{\delta\downarrow0}
    i\int_\delta^s
    e^{i(s-\sigma)L_\nu}
    L_\nu W_\sigma
    \,d\sigma
    \quad\hbox{in }L^2
  \end{equation}
  (which exists by retarded Strichartz estimates).
  Then there are $c_0>0$ and $s_0>0$ such that
  \begin{equation}
    \label{eq:rate-linear-response-lower-bound}
    \|Y_{\mathrm{lin}}(s)\|_2
    \geq
    c_0s^{\nu+1/2}\log\frac1s,
    \qquad
    0<s<s_0.
  \end{equation}
\end{lemma}

\begin{proof}
  Set
  \begin{equation*}
    \ell_\sigma=\log\frac1{4\sigma}.
  \end{equation*}
  Since $|f|=|c|\chi r^\nu$, we can write
  \begin{equation}
    \label{eq:rate-W-leading-trace}
    W_\sigma
    =
    f
    -ia|c|c
    \chi^2r^{2\nu}
    \log\frac1{4\sigma}
    +
    Q_\sigma
  \end{equation}
  where $Q_\sigma=Q_{\ell_\sigma}$ in the notation of
  Lemma~\ref{lem:rate-remainder-hankel-bound}.
  By linearity, the response splits as
  \begin{equation*}
    Y_{\mathrm{lin}}=Y_f+Y_{\mathrm{lead}}+Y_Q,
  \end{equation*}
  where the three terms are generated by each of the three
  terms on the right-hand side of 
  \eqref{eq:rate-W-leading-trace}, respectively.

  The term $Y_{f}$ is harmless: since $f\in\dom(L_\nu)$,
  \begin{equation*}
    \|Y_f(s)\|_2
    \leq
    \int_0^s\|L_\nu f\|_2\,d\sigma
    =s\|L_\nu f\|_2.
  \end{equation*}
  Consider $Y_{\mathrm{lead}}$.
  Let $\Pi_s$ be the spectral projection of $L_\nu$ onto
  $[s^{-1},4s^{-1}]$, that is
  \begin{equation*}
    \mathcal H_\nu\Pi_s\mathcal H_\nu^{-1} =\mathbf 1_{I_s},
    \qquad
    I_s=[s^{-1/2},2s^{-1/2}].
  \end{equation*}
  Put $g(r)=\chi(r)^2r^{2\nu}$.  After applying $\Pi_s$, the
  Duhamel integral is an ordinary $L^2$ integral.  Its Hankel
  representation is
  \begin{equation*}
    \bigl(\mathcal H_\nu\Pi_sY_{\mathrm{lead}}(s)\bigr)(\rho)
    =
    \mathbf 1_{I_s}(\rho)
    a|c|c\,(\mathcal H_\nu g)(\rho)
    \rho^2\int_0^s
      e^{i(s-\sigma)\rho^2}\ell_\sigma\,d\sigma.
  \end{equation*}
  The logarithmic identity
  \eqref{eq:rate-logarithmic-Duhamel-symbol} gives, uniformly
  on this spectral band,
  \begin{equation*}
    \rho^2\int_0^s
      e^{i(s-\sigma)\rho^2}\ell_\sigma\,d\sigma
    =
    \frac{e^{is\rho^2}-1}{i}\log\frac1s+O(1).
  \end{equation*}
  Indeed, $s\rho^2\in[1,4]$ there; the second term in
  \eqref{eq:rate-logarithmic-Duhamel-symbol} and the constant
  shift from $\log(1/\sigma)$ to $\ell_\sigma$ are uniformly
  bounded.  The Hankel threshold lemma, applied to $g$, gives
  \begin{equation*}
    (\mathcal H_\nu g)(\rho)
    =
    \kappa_\nu\rho^{-2\nu-2}
    +o(\rho^{-2\nu-2}).
  \end{equation*}
  Since $\kappa_\nu\neq0$ and
  $|e^{is\rho^2}-1|$ is bounded below on $s\rho^2\in[1,4]$,
  scaling $y=s^{1/2}\rho$ gives
  \begin{equation}
    \label{eq:rate-projected-leading-response}
    \|\Pi_sY_{\mathrm{lead}}(s)\|_2
    =
    C_{a,c,\nu}
    s^{\nu+1/2}\log\frac1s
    (1+o(1)),
    \qquad
    C_{a,c,\nu}>0.
  \end{equation}

  It remains to show that the remainder $Y_{Q}$ cannot cancel
  this contribution.  Since $Q_\ell=O(r^{3\nu})$ near the origin
  and is cutoff at infinity, $Q_\ell\in L^2$.  Thus the ordinary
  spectral projection $\Pi_sQ_\ell$ is defined and belongs to
  $\dom(L_\nu)$, even though $Q_\ell\not\in\dom(L_\nu)$
  in general.  Spectral calculus gives directly
  \begin{equation*}
    \Pi_sL_\nu\Pi_sQ_{\ell_\sigma}
    =
    L_\nu\Pi_sQ_{\ell_\sigma},
    \qquad
    \bigl(\mathcal H_\nu L_\nu\Pi_sQ_{\ell_\sigma}\bigr)(\rho)
    =
    \mathbf 1_{I_s}(\rho)\,\rho^2
    (\mathcal H_\nu Q_{\ell_\sigma})(\rho),
    \qquad
    I_s=[s^{-1/2},2s^{-1/2}].
  \end{equation*}
  This is the localized form of the distributional quantity
  $L_\nu Q_\ell$; no additional boundary term occurs because
  $Q_\ell=O(r^{3\nu})$, $Q_\ell'=O((1+|\ell|)^2r^{3\nu-1})$,
  while Friedrichs test functions satisfy
  $v=O(r^\nu)$ and $v'=O(r^{\nu-1})$ at the origin.  
  The Hankel bound in Lemma~\ref{lem:rate-remainder-hankel-bound} 
  therefore gives,
  after commuting the projection with the propagator,
  \begin{equation*}
    \Pi_sY_Q(s)
    =
    i\int_0^s e^{i(s-\sigma)L_\nu}
    L_\nu\Pi_sQ_{\ell_\sigma}\,d\sigma,
  \end{equation*}
  and
  \begin{equation*}
    \left\|L_\nu\Pi_sQ_{\ell_\sigma}\right\|_2^2
    =
    \int_{I_{s}}
      \left|\rho^2(\mathcal H_\nu Q_{\ell_\sigma})(\rho)\right|^2
    \rho\,d\rho 
    \leq
    C(1+\ell_\sigma)^{2M}
    \int_{I_{s}}\rho^{1-6\nu}\,d\rho
    \leq
    C(1+\ell_\sigma)^{2M}s^{3\nu-1}.
  \end{equation*}
  Minkowski's inequality and
  $\int_0^s(1+\ell_\sigma)^M\,d\sigma
  \lesssim s(1+\log(1/s))^M$ then give
  \begin{equation}
    \label{eq:rate-projected-remainder-bound}
    \left\|\Pi_sY_Q(s)\right\|_2
    \leq
    C s^{(3\nu+1)/2}
    \left(1+\log\frac1s\right)^M
    =
    o\left(
      s^{\nu+1/2}\log\frac1s
    \right).
  \end{equation}
  The contribution of $f$ is negligible as well, because
  $s=o(s^{\nu+1/2}\log(1/s))$ for $0<\nu<1/2$.  Therefore
  \begin{equation*}
    \|Y_{\mathrm{lin}}(s)\|_2
    \geq
    \|\Pi_sY_{\mathrm{lin}}(s)\|_2 
    \geq
    \|\Pi_sY_{\mathrm{lead}}(s)\|_2
    -\|\Pi_sY_Q(s)\|_2
    -\|Y_f(s)\|_2
    \geq
    c_0s^{\nu+1/2}\log\frac1s
  \end{equation*}
  for all sufficiently small $s$, proving
  \eqref{eq:rate-linear-response-lower-bound}.
\end{proof}

\begin{proof}[Proof of
  Theorem~\ref{the:pure-trace-rate-sharpness}]
  Use the value $b_0=1/2+\nu/2$ fixed in the theorem 
  (any fixed $1/2<b_{0}<1/2+\nu$ would work equally well) and put
  $z=u-u_{\mathrm{ap}}$.
  Suppose, for a contradiction, that $z$ satisfies
  \eqref{eq:main-final-state-bound} for some
  $b\geq1/2+\nu$.

  By the fixed point relation \eqref{eq:fixed-point-map},
  we can split $z=z_{\mathrm{lin}}+z_{\mathrm{nl}}$ with
  \begin{equation*}
    z_{\mathrm{lin}}=
    -i\int_{t}^{\infty}
    e^{-i(t-s)H_\alpha}
    R(s)ds,
    \qquad
    z_{\mathrm{nl}}=
    i\int_{t}^{\infty}
    e^{-i(t-s)H_\alpha}
    \bigl[
      \lambda
      \bigl(
        N(u_{\mathrm{ap}}+z)
        -
        N(u_{\mathrm{ap}})
      \bigr)
    \bigr]\,ds.
  \end{equation*}
  Recall also that
  \begin{equation*}
    R(t)
    =
    -\frac1{4t^2}M(t)D(t)H_\alpha W(t).
  \end{equation*}
  The profile $W$ and the residual $R$ lie in the
  $m_\nu$-channel.  Since the linear flow, $M(t)$, and $D(t)$
  preserve angular channels, so do $z_{\mathrm{lin}}$ and the
  rescaled profile below.  Suppressing the common factor
  $e^{im_\nu\theta}$, the operator $H_\alpha$ reduces to
  $L_\nu$ and
  \begin{equation*}
    z_{\mathrm{lin}}(t)
    =
    -i\int_t^\infty
      e^{-i(t-\tau)L_\nu}R(\tau)\,d\tau,
    \qquad
    R(t)
    =
    -\frac1{4t^2}M(t)D(t)L_\nu W(t)
  \end{equation*}
  where we reuse the notation $R(t)$ for its radial component.
  The integral is understood in the retarded Strichartz sense.

  Put $\sigma=\sigma(t)=(4t)^{-1}$ and define
  \begin{equation*}
    Y(\sigma)
    =
    D(t)^{-1}M(t)^{-1}z_{\mathrm{lin}}(t),
    \qquad
    z_{\mathrm{lin}}(t)=M(t)D(t)Y(\sigma(t)).
  \end{equation*}
  Since $M(t)D(t)$ is unitary,
  \begin{equation*}
    \|z_{\mathrm{lin}}(t)\|_2=\|Y(\sigma(t))\|_2.
  \end{equation*}
  We now change the time variable in the integral
  for $z_{\mathrm{lin}}$.
  For a pole regularized source and a finite terminal time
  $S>t$, set $\eta=(4\tau)^{-1}$ 
  and use the channel pseudoconformal intertwining for
  $L_\nu$.  The mixed retarded estimate then allows us to remove
  the pole regularization and let $S\to\infty$, yielding
  \begin{equation}
    \label{eq:rate-response-limit}
    Y(\sigma)
    =
    \lim_{\delta\downarrow0}
    i\int_\delta^\sigma
    e^{i(\sigma-\eta)L_\nu}
    L_\nu W_\eta\,d\eta
    \quad\hbox{in }L^2.
  \end{equation}
  Here, with $a=\lambda/2$,
  \begin{equation*}
    W_\eta(r)
    =
    f(r)\exp\left(-ia|f(r)|\log\frac1{4\eta}\right).
  \end{equation*}
  The retarded Strichartz estimate gives a meaning to this formula 
  and makes its right hand side convergent in $L^2$.  Thus $Y$ is
  exactly the response appearing in
  Lemma~\ref{lem:rate-pure-trace-response} and we get
  \begin{equation}
    \label{eq:rate-physical-response-lower-bound}
    \|z_{\mathrm{lin}}(t)\|_2
    \geq
    c_1t^{-\nu-1/2}\log t
  \end{equation}
  for all sufficiently large $t$.

  We next estimate $z_{\mathrm{nl}}$.  The pointwise quadratic
  Lipschitz bound for $N(w)=|w|w$ gives
  \begin{equation*}
    |N(u_{\mathrm{ap}}+z)-N(u_{\mathrm{ap}})|
    \lesssim
    |u_{\mathrm{ap}}||z|+|z|^2.
  \end{equation*}
  The contradiction hypothesis is exactly the statement that
  $z\in X_{b,T}$.  Applying the retarded Strichartz estimate and
  then using
  \eqref{eq:linear-long-range-term} and
  \eqref{eq:quadratic-remainder-term} (with $|\lambda|$
  absorbed into the constant) yields
  \begin{equation}
    \label{eq:rate-nonlinear-response-upper-bound}
    \|z_{\mathrm{nl}}(t)\|_2
    \leq
    C
    \left(
      t^{-b}\|z\|_{X_{b,T}}
      +
      t^{1/2-2b}\|z\|_{X_{b,T}}^2
    \right).
  \end{equation}
  Since $b\geq1/2+\nu$, both terms are
  $o(t^{-\nu-1/2}\log t)$.  Combining
  \eqref{eq:rate-physical-response-lower-bound} and
  \eqref{eq:rate-nonlinear-response-upper-bound} gives
  \begin{equation*}
    \|z(t)\|_2
    \geq
    \|z_{\mathrm{lin}}(t)\|_2
    -\|z_{\mathrm{nl}}(t)\|_2
    \geq
    \frac{c_1}{2}
    t^{-\nu-1/2}\log t
  \end{equation*}
  for large $t$, contradicting
  \eqref{eq:main-final-state-bound}.

  Finally, suppose that another solution $v$ satisfied the
  forbidden bound.  Since $b\geq1/2+\nu>b_0$, it also belongs
  to the final state class with exponent $b_0$ on a common
  tail.  The uniqueness assertion in
  Theorem~\ref{the:main-final-state} then gives $v=u$, which
  contradicts the lower bound just proved.
\end{proof}

\subsection{The first correction}
\label{sec:second-order}

We next resolve the first term hidden in the error
\eqref{eq:intro-error}.  Cazenave--Naumkin~\cite{CN18}
derive higher asymptotic expansions for the flat critical
equation under a strong nonvanishing hypothesis on
oscillatory initial profiles; see also
Masaki--Miyazaki--Uriya~\cite{MMU19} for a candidate second
profile in a related final state problem, and
Jendrej--Salvi~\cite{JS26} for arbitrary order expansions in
a one dimensional polynomial problem with a cubic long range
component.  The quadratic map is not smooth enough at
transverse zeros for a direct high order Sobolev iteration.
We therefore control the profile modulus and phase
separately; the vanishing of the modulus of high enough order
allows to control the derivatives of the phase, which are
singular at the pole.

\begin{definition}[Second order core profile]
\label{def:second-order-core}
  A profile $\phi$ belongs to the \emph{second order core} if, on
  $\mathbb R^2\setminus\{0\}$,
  \begin{equation}
    \label{eq:second-core-factorization}
    \phi=\rho\omega,
  \end{equation}
  where $\rho\in C_c^\infty(\mathbb R^2)$ is nonnegative,
  \begin{equation}
    \label{eq:rho-sixth-order}
    \partial^\beta\rho(0)=0
    \qquad
    (|\beta|\leq5),
  \end{equation}
  and
  $\omega\in C^\infty(\mathbb R^2\setminus\{0\};\mathbb C)$
  satisfies $|\omega(x)|=1$ and
  \begin{equation}
    \label{eq:omega-symbol-bounds}
    |\partial^\beta\omega(x)|
    \leq
    C_\beta |x|^{-|\beta|}
    \qquad
    (|\beta|\leq6).
  \end{equation}
\end{definition}

This class meets the pole and is not restricted to one
channel.  For example, $\omega=e^{im\theta}$ is allowed, while
$\rho$ may depend on both $r$ and $\theta$.  
Every such profile belongs to $\mathcal Y$ and is
therefore admissible by Lemma~\ref{lem:concrete-profiles}.

Put
\begin{equation}
  \label{eq:second-order-a}
  a
  =
  \frac{\lambda}{2}\rho
\end{equation}
and define the real linear map
\begin{equation}
  \label{eq:second-order-linearization}
  \mathcal L_\phi h
  =
  a\omega\operatorname{Re}(\overline\omega h).
\end{equation}
Away from the pole, the right hand side is meaningful also
where $\rho=0$.  At the pole it is set equal to zero; the
coefficient extends there by
\eqref{eq:rho-sixth-order}--\eqref{eq:omega-symbol-bounds}.
Set
\begin{equation}
  \label{eq:second-order-B}
  B_\phi
  =
  -iI-\mathcal L_\phi.
\end{equation}
If $h=\omega(x+iy)$, then $B_{\phi}h=-(a+i)x\omega+y \omega$
which can be written
\begin{equation}
  \label{eq:B-real-matrix}
  B_\phi h
  =
  \omega
  \begin{pmatrix}
    -a&1\\
    -1&0
  \end{pmatrix}
  \begin{pmatrix}
    x\\
    y
  \end{pmatrix}.
\end{equation}
The determinant is one.  Thus $B_\phi$ is invertible
pointwise, and
\begin{equation}
  \label{eq:B-explicit-inverse}
  B_\phi^{-1}
  \bigl(\omega(u+iv)\bigr)
  =
  \omega
  \bigl(
    -v+i(u-av)
  \bigr).
\end{equation}

Define
\begin{equation}
  \label{eq:second-order-R}
  R_0
  =
  H_\alpha\phi,
  \qquad
  R_1
  =
  i\left(
    2\nabla a\cdot\nabla_\alpha\phi
    +
    (\Delta a)\phi
  \right),
  \qquad
  R_2
  =
  |\nabla a|^2\phi.
\end{equation}
The coefficients $q_2,q_1,q_0$ are fixed recursively by
\begin{equation}
  \label{eq:second-order-q-recursion}
    B_\phi q_2=
    \frac14R_2,
    \qquad
    B_\phi q_1=
    \frac14R_1-2iq_2,
    \qquad
    B_\phi q_0=
    \frac14R_0-iq_1.
\end{equation}
With $q(\sigma) = q_2\sigma^2+q_1\sigma+q_0,$
as in \eqref{eq:q-polynomial}, set
\begin{equation}
  \label{eq:second-order-profile}
  W_2(t)
  =
  e^{-ia\log t}
  \left(
    \phi+\frac1tq(\log t)
  \right).
\end{equation}
Then \eqref{eq:second-order-uap} is equivalently
$u_{\mathrm{ap}}^{(2)}(t)=M(t)D(t)W_2(t)$.

\begin{proposition}[Corrected residual]
\label{prop:second-order-residual}
  Let $\phi$ be a second order core profile.  Then
  \begin{equation*}
    W_2
    \in
    C([1,\infty);\dom(H_\alpha))
    \cap
    C^1([1,\infty);L^2),
  \end{equation*}
  where the first space has the graph topology.  Moreover,
  $u_{\mathrm{ap}}^{(2)}(t)=M(t)D(t)W_2(t)$ satisfies
  \begin{equation}
    \label{eq:second-order-residual}
    \left\|
      (i\partial_t-H_\alpha)u_{\mathrm{ap}}^{(2)}
      -
      \lambda
      |u_{\mathrm{ap}}^{(2)}|
      u_{\mathrm{ap}}^{(2)}
    \right\|_2
    \leq
    C_\phi t^{-3}(1+\log t)^4.
  \end{equation}
\end{proposition}

\begin{proof}
  Write $\sigma=\log t$ and
  \begin{equation}
    \label{eq:F-lambda}
    F_\lambda(z)
    =
    \frac{\lambda}{2}|z|z.
  \end{equation}
  We first perform the algebra on the punctured plane.  The
  graph and time regularity needed for the pseudoconformal
  identity are verified below.
  The pseudoconformal identity and nonlinear scaling give
  as before, for $V=V(\sigma,x)$,
  \begin{equation}
    \label{eq:log-profile-equation}
    (i\partial_t-H_\alpha)M(t)D(t)V
    -
    \lambda|M(t)D(t)V|M(t)D(t)V
    =
    \frac1tM(t)D(t)
    \left[
      i\partial_\sigma V
      -
      F_\lambda(V)
      -
      \frac{e^{-\sigma}}4H_\alpha V
    \right].
  \end{equation}
  The leading profile $W=e^{-ia\sigma}\phi$ solves
  $i\partial_\sigma W=F_\lambda(W)$ and leaves a remainder
  of order $t^{-2}$ at the RHS of \eqref{eq:log-profile-equation}.
  To remove this error, consider a corrected profile
  \begin{equation*}
    W_2=W+e^{-\sigma}Y,
    \qquad
    W=e^{-ia\sigma}\phi,
    \qquad
    Y(\sigma)=e^{-ia\sigma}q(\sigma).
  \end{equation*}
  Denote by $\mathcal E[V]$ the quantity in square brackets
  in \eqref{eq:log-profile-equation}.
  If we plug $V=W_{2}$ inside $\mathcal E[V]$ we obtain
  \begin{equation*}
    \mathcal E[W_{2}]
    =
    e^{-\sigma}
    \left[
      iY'-iY
      -
      DF_\lambda(W)Y
      -
      \frac14H_\alpha W
    \right]
    -
    e^{-2\sigma}
    \left[
      \mathcal Q_\lambda(\sigma)
      +
      \frac14H_\alpha Y
    \right]
  \end{equation*}
  where
  \begin{equation*}
    \mathcal Q_\lambda(\sigma)
    :=
    e^{2\sigma}
    \left[
      F_\lambda(W+e^{-\sigma}Y)
      -
      F_\lambda(W)
      -
      e^{-\sigma}DF_\lambda(W)Y
    \right].
  \end{equation*}

  The first term in $\mathcal{E}[W_{2}]$
  cancels exactly with our choice of $q(\sigma)$.
  Indeed, by gauge covariance
  \begin{equation*}
    DF_\lambda(W)Y
    =
    e^{-ia\sigma}DF_\lambda(\phi)q
    =
    e^{-ia\sigma}
    \bigl(aq+\mathcal L_\phi q\bigr),
  \end{equation*}
  while
  \begin{equation*}
    iY'-iY
    =
    e^{-ia\sigma}
    \bigl(iq'+aq-iq\bigr).
  \end{equation*}
  Since $B_\phi=-iI-\mathcal L_\phi$, we obtain
  \begin{equation*}
    iY'-iY-DF_\lambda(W)Y
    =
    e^{-ia\sigma}
    \bigl(iq'+B_\phi q\bigr).
  \end{equation*}
  Moreover, with $R_{j}$ as in \eqref{eq:second-order-R},
  \begin{equation*}
    H_\alpha W
    =
    e^{-ia\sigma}
    \bigl(R_0+\sigma R_1+\sigma^2R_2\bigr).
  \end{equation*}
  Summing up
  \begin{equation*}
    \mathcal E[W_{2}]
    =
    e^{-\sigma}e^{-ia\sigma}
    \left[
      iq'+B_\phi q
      -
      \frac14
      \bigl(R_0+\sigma R_1+\sigma^2R_2\bigr)
    \right]
    -
    e^{-2\sigma}
    \left[
      \mathcal Q_\lambda(\sigma)
      +
      \frac14H_\alpha Y
    \right].
  \end{equation*}
  By conditions \eqref{eq:second-order-q-recursion}, the first
  term is identically 0, and only the second term survives.
  Since $t^{-1}=e^{-\sigma}$, the physical residual is
  \begin{equation}\label{eq:phys-resid}
      (i\partial_t-H_\alpha)u_{\mathrm{ap}}^{(2)}
      -
      \lambda
      |u_{\mathrm{ap}}^{(2)}|
      u_{\mathrm{ap}}^{(2)}
      =
      -
      t^{-3}M(t)D(t)
      \left[
        \mathcal Q_\lambda(\log t)
        +
        \frac14H_\alpha Y(\log t)
      \right].
  \end{equation}

  To justify these computations, 
  we check the regularity at the pole.  From
  \eqref{eq:rho-sixth-order} and
  \eqref{eq:omega-symbol-bounds},
  \begin{equation}
    \label{eq:phi-pole-orders}
    |\partial^\beta\phi(x)|
    \leq
    C_\beta |x|^{6-|\beta|}
    \qquad
    (|\beta|\leq6,\ |x|\leq1).
  \end{equation}
  The pole orders can be read directly from the formulas.
  For $|\beta|\leq2$, the terms $R_0$, $R_1$, and $R_2$ have
  orders $4-|\beta|$, $10-|\beta|$, and
  $16-|\beta|$, respectively.  Differentiating the explicit
  inverse in \eqref{eq:B-explicit-inverse} costs at most one
  power of $|x|$ for each derivative of $\omega$.  The
  recursion therefore gives $q_2$ of order $16$, $q_1$ of
  order $10$, and $q_0$ of order $4$.  In particular,
  \begin{equation}
    \label{eq:q-pole-orders}
    |\partial^\beta q_j(x)|
    \leq
    C_\beta |x|^{4-|\beta|}
    \qquad
    (|\beta|\leq2,\ |x|\leq1).
  \end{equation}
  All these functions are compactly supported.  The estimates
  imply, by the same cutoff argument used for
  Lemma~\ref{lem:concrete-profiles}, that
  $e^{-ia\sigma}q(\sigma)\in\mathcal Y$.  They also give
  \begin{equation}
    \label{eq:q-growth-bounds}
    \|q(\sigma)\|_\infty
    +
    \|q(\sigma)\|_4
    \leq
    C_\phi(1+\sigma)^2
  \end{equation}
  and
  \begin{equation}
    \label{eq:Hq-growth-bound}
    \left\|
      H_\alpha
      \bigl(e^{-ia\sigma}q(\sigma)\bigr)
    \right\|_2
    \leq
    C_\phi(1+\sigma)^4.
  \end{equation}
  Hence $W_2(t)\in\mathcal Y\subset\dom(H_\alpha)$.
  The same estimates applied to differences in $\sigma$ give
  graph continuity.  If $\sigma=\log t$, direct
  differentiation gives
  \begin{equation*}
    \partial_tW_2(t)
    =
    \frac1t e^{-ia\sigma}
    \left[
      -ia\phi
      +
      \frac1t
      \bigl(q'(\sigma)-(1+ia)q(\sigma)\bigr)
    \right].
  \end{equation*}
  The right hand side is continuous in $L^2$.  This proves the
  asserted time regularity and justifies the pseudoconformal
  identity used above.

  We finally estimate the surviving error. 
  Since $D F_\lambda$ is globally Lipschitz,
  \begin{equation}
    \label{eq:F-second-difference}
    |F_\lambda(z+h)-F_\lambda(z)
      -D F_\lambda(z)h|
    \leq
    C|\lambda||h|^2.
  \end{equation}
  By \eqref{eq:F-second-difference},
  \eqref{eq:q-growth-bounds}, and
  \eqref{eq:Hq-growth-bound}, the bracket in 
  \eqref{eq:phys-resid} has $L^2$ norm at
  most $C_\phi(1+\log t)^4$.  Unitarity of $M(t)D(t)$ proves
  \eqref{eq:second-order-residual}.
\end{proof}

\begin{proof}[Proof of Theorem~\ref{the:second-order}]
  Repeat the final state contraction with
  $u_{\mathrm{ap}}^{(2)}$ in place of
  $u_{\mathrm{ap}}$.  The weighted quadratic estimate only
  requires $b>1/2$.  From \eqref{eq:q-growth-bounds},
  \begin{equation}
    \label{eq:second-uap-infinity}
    \|u_{\mathrm{ap}}^{(2)}(t)\|_\infty
    \leq
    \frac{\|\rho\|_\infty}{2t}
    +
    C_\phi t^{-2}(1+\log t)^2.
  \end{equation}
  The first term is absorbed by smallness of
  $\|\rho\|_\infty$.  The second contributes a coefficient
  $\eta_T$ to the contraction estimate, where
  \begin{equation}
    \label{eq:eta-T-bound}
    \eta_T
    \leq
    C_{\phi,b,\lambda}
    T^{-1}(1+\log T)^2.
  \end{equation}
  If $\mathcal R_2$ denotes the residual in
  \eqref{eq:phys-resid}, then
  \begin{equation}
    \label{eq:second-residual-tail}
    \sup_{\tau\geq T}
    \tau^b
    \|\mathcal R_2\|_{
      L^1([\tau,\infty);L^2)}
    \leq
    C_{\phi,b}
    T^{b-2}(1+\log T)^4.
  \end{equation}
  This tends to zero because $b<2$.  Thus the analogue of
  \eqref{eq:fixed-point-bound} is
  \begin{equation}
    \label{eq:second-fixed-point-bound}
    \|\Phi_2(z)\|_{X_{b,T}}
    \leq
    C_\alpha
    \biggl[
      \left(
        \frac{|\lambda|\|\rho\|_\infty}{b}
        +
        \eta_T
      \right)
      \|z\|_{X_{b,T}}
      {}+
      |\lambda|T^{1/2-b}
      \|z\|_{X_{b,T}}^2
      +
      C_{\phi,b}
      T^{b-2}(1+\log T)^4
    \biggr].
  \end{equation}
  The difference estimate has the same linear coefficient.
  At this point, the proof proceeds exactly as for
  Theorem \ref{the:all-channel}, and we omit it.
\end{proof}

\begin{remark}[Relation with the first order solution]
\label{rem:second-first-identification}
  Fix once and for all $b'\in(1/2,1)$ and choose the
  smallness threshold in the previous proof no larger than
  $\varepsilon_{b',\alpha,\lambda}$ from
  Theorem~\ref{the:all-channel}.  Since $b>1>b'$ and
  $u_{\mathrm{ap}}^{(2)}-u_{\mathrm{ap}}\in X_{b',T}$,
  the solution constructed here belongs to the uniqueness class
  of Theorem~\ref{the:all-channel}.
  We conclude that
  Theorem~\ref{the:second-order} refines the same modified wave
  operator and does not construct a new solution.
\end{remark}

\begin{remark}[Cubic perturbation]
\label{rem:second-order-cubic}
  For a more general NLS with the additional term
  $\kappa|u|^2u$, the logarithmic profile is unchanged.
  In \eqref{eq:log-profile-equation}, the order
  $e^{-\sigma}$ forcing becomes
  \begin{equation}
    \label{eq:cubic-second-forcing}
    \frac{e^{-\sigma}}4
    \left(
      H_\alpha V+\kappa|V|^2V
    \right).
  \end{equation}
  Thus the same recursion applies after replacing
  $R_0$ by
  \begin{equation}
    \label{eq:cubic-R0}
    R_0+\kappa|\phi|^2\phi.
  \end{equation}
  The forward final state argument is unchanged.  A global
  statement for the mixed equation additionally depends on
  the sign and size assumptions in its cubic Cauchy theory.
\end{remark}

\appendix

\section{Single channel factorization}
\label{sec:one-channel}

We give a direct channel computation that fixes the
normalization of the distorted transform.
We also check the channel invariance behind
Corollary~\ref{cor:natural-channel-final-states}.

\subsection{Direct Hankel factorization}
\label{sec:channel-factorization}

\begin{lemma}[One channel factorization]
\label{lem:channel-factorization}
  Let $\mu\geq0$, $t>0$, and
  $U_\mu(t)=e^{-itL_\mu}$.  Then
  \begin{equation}
    \label{eq:channel-factorization}
    U_\mu(t)
    =
    e^{-i\pi\mu/2}
    M(t)D(t)\mathcal H_\mu M(t),
  \end{equation}
  where $M(t)$ and $D(t)$ act on the radial variable as in
  \eqref{eq:MD-definition-intro}.
\end{lemma}

\begin{proof}
  The spectral representation of $L_\mu$ gives
  \begin{equation}
    \label{eq:spectral-double-integral}
    U_\mu(t)f(r)
    =
    \int_0^\infty
    \int_0^\infty
    e^{-it\rho^2}
    J_\mu(r\rho)J_\mu(s\rho)
    f(s)s\,ds\,\rho\,d\rho.
  \end{equation}
  Weber's second exponential integral, first with a positive
  damping and then by continuation, gives
  \begin{equation}
    \label{eq:Weber}
    \int_0^\infty
    e^{-it\rho^2}
    J_\mu(r\rho)J_\mu(s\rho)\rho\,d\rho
    =
    \frac{e^{-i\pi\mu/2}}{2it}
    \exp\left(
      \frac{i(r^2+s^2)}{4t}
    \right)
    J_\mu\left(\frac{rs}{2t}\right).
  \end{equation}
  Substitution in \eqref{eq:spectral-double-integral}
  yields \eqref{eq:channel-factorization}:
  \begin{equation}
    \label{eq:Weber-factorized}
    U_\mu(t)f(r)
    =
    e^{-i\pi\mu/2}
    \frac{e^{ir^2/(4t)}}{2it}
    \left[
      \mathcal H_\mu\bigl(M(t)f\bigr)
    \right]\left(\frac r{2t}\right).
    \qedhere
  \end{equation}
\end{proof}

\subsection{Channel invariance}

The gauge invariant nonlinearity preserves each channel:
\begin{equation}
  \label{eq:channel-nonlinear}
  |e^{im\theta}v|e^{im\theta}v
  =
  e^{im\theta}|v|v.
\end{equation}
Thus \eqref{eq:main-NLS} on the subspace
$e^{im\theta}L^2(r\,dr)$ reduces to
\begin{equation}
  \label{eq:radial-NLS}
  i\partial_tv
  =
  L_{\mu_m}v+\lambda|v|v.
\end{equation}
Every operation in the fixed point map
\eqref{eq:fixed-point-map} preserves the subspace
$e^{im\theta}L^2(r\,dr)$.  The linear flow does so by
\eqref{eq:channel-linear}, the approximate state belongs to
the same channel, and \eqref{eq:channel-nonlinear} treats the
nonlinearity.  Thus the solution in
Corollary~\ref{cor:natural-channel-final-states} remains in its
prescribed channel for all times.
Lemma~\ref{lem:channel-factorization} gives the direct
radial normalization of the asymptotic formula.

\begingroup
\sloppy
\bibliographystyle{amsplain}
\bibliography{quadrNLS-AB-modWO}
\endgroup
\end{document}